\documentclass[11pt]{amsart}
\usepackage[english]{babel}

\usepackage{latexsym}
\usepackage{amssymb}
\usepackage{amsmath}
\usepackage{amsfonts}
\usepackage[mathscr]{eucal}
\usepackage{color}

\usepackage{amscd}

\usepackage{graphicx}
\usepackage{epsfig}
\usepackage{relsize}

\usepackage[table,xcdraw]{xcolor}
\usepackage{longtable}
\usepackage{rotating}
\usepackage{array}
\usepackage{multicol}
\usepackage{multirow}
\usepackage{makecell}
\usepackage{mathtools}
  \usepackage{cellspace}
\usepackage{booktabs}

\usepackage{amsthm}
\usepackage[all,cmtip]{xy}
\usepackage[pdfstartview=FitH]{hyperref}
\usepackage{pst-grad} % For gradients
\usepackage{pst-plot} % For axes
\usepackage{comment}
\usepackage{psfrag}
\usepackage{verbatim}
\usepackage{hyperref}
\hypersetup{colorlinks,allcolors=blue}
\usepackage{pifont}% http://ctan.org/pkg/pifont

\usepackage{soul}
 \usepackage[normalem]{ulem}

\usepackage{lineno}
\newtheorem{thm}{Theorem}[section]
\newtheorem*{thm*}{Theorem}
\newtheorem{cor}[thm]{Corollary}	% COROLLARY
\newtheorem*{cor*}{Corollary}		% COROLLARY NO NUMBER
\newtheorem{lem}[thm]{Lemma}
\newtheorem{prop}[thm]{Proposition}

\newtheorem{theorem}{Theorem}

\theoremstyle{definition}
\newtheorem*{definition*}{Definition}    	% DEF NO NUMBER

\newtheorem{remark}[theorem]{Remark}
\newtheorem{case}{Case}[section]

\newtheorem{subcase}{Case}[case]

\newtheorem*{ack}{Acknowledgements}

\numberwithin{equation}{section}

\DeclareMathOperator{\Susp}{Susp}
\DeclareMathOperator{\curv}{curv}
\DeclareMathOperator{\seif}{Seif}
\DeclareMathOperator{\diam}{diam}
\DeclareMathOperator{\vol}{vol}

\newcommand{\RR}{\mathbb{R}}
\newcommand{\CC}{\mathbb{C}}

\newcommand{\RP}{\mathbb{R}P}

\newcommand{\Real}{\mathrm{\mathbb{R}}}
\newcommand{\bpt}{\mathrm{B(pt)}}
\newcommand{\bstwo}{\mathrm{B}(S_2)}
\newcommand{\bsfour}{\mathrm{B}(S_4)}
\newcommand{\bptwo}{\mathrm{B}(\mathbb{R}P^2)}
\newcommand{\mo}{\mathrm{Mo}}
\newcommand{\SL}{\widetilde{\mathrm{SL}_2(\mathbb{R})}}

\newcommand{\Nil}{\mathrm{Nil}}
\newcommand{\Sol}{\mathrm{Sol}}

\newcommand{\eGHto}{\stackrel {_\mathrm{eGH}}{\longrightarrow} }
\newcommand{\GHto}{\stackrel { _\textrm{GH}}{\longrightarrow} }
\newcommand{\csum}{\operatornamewithlimits{\#}}

\begin{document}

%------------------------------------------------------
% BEGIN FRONT MATTER -------------------------------------------------------
%------------------------------------------------------

% TITLE

\title[Sufficiently collapsed Alexandrov $3$-spaces]{Sufficiently collapsed irreducible Alexandrov $3$-spaces are geometric \\ 
}

% AUTHORS

\author[F.~Galaz-Garc\'ia]{Fernando Galaz-Garc\'ia$^{\ast}$}

\author[L.~Guijarro]{Luis Guijarro$^{\ast\ast}$}

\author[J.~N\'u\~nez-Zimbr\'on]{Jes\'us N\'u\~nez-Zimbr\'on $^{\ast\ast\ast}$}

%SUPPORT ACKNOWLEDGMENT

\thanks{$^*$Supported in part by research grants MTM2011-22612, MTM2014-57769-C3-3-P from the Ministerio de Econom\'ia y Competitividad, MINECO: ICMAT Severo Ochoa project SEV-2011-0087, and by the Deutsche Forschungsgemeinschaft grant GA 2050 2-1 within the  Priority Program SPP 2026 ``Geometry at Infinity''.} 

\thanks{$^{\ast\ast}$Supported by research grants MTM2011-22612, MTM2014-57769-C3-3-P from the Ministerio de Econom\'ia y Competitividad and MINECO: ICMAT Severo Ochoa project SEV-2011-0087.}

\thanks{$^{\ast\ast\ast}$Supported by UCMEXUS-CONACYT postdoctoral grant  ``Alexandrov Geometry''.}

% ADDRESSES

\address[F.~Galaz-Garc\'ia]{Institut f\"ur Algebra und Geometrie, Karlsruher Institut f\"ur Technologie (KIT), Germany}
\email{galazgarcia@kit.edu}

\address[L.~Guijarro]{Department of Mathematics, Universidad Aut\'onoma de Madrid and ICMAT CSIC-UAM-UC3M, Spain}
\email{luis.guijarro@uam.es} 

\address[J.~N\'u\~nez-Zimbr\'on]{Department of Mathematics, University of California, Santa Barbara, CA 93106}
\email{zimbron@ucsb.edu}

% DATE

\date{\today}

% MATH SUBJECT CLASSIFICATION AND KEYWORDS


% ABSTRACT

\begin{abstract}
We prove that sufficiently collapsed, closed and irreducible three-dimensional Alexandrov spaces are modeled on one of the  three-dimensional Thurston geometries, excluding the hyperbolic one. This extends a result of Shi\-o\-ya and Ya\-ma\-gu\-chi, 
originally formulated for Riemannian manifolds, to the Alexandrov setting. 
\end{abstract}
\setcounter{tocdepth}{1}

\maketitle

\tableofcontents

%------------------------------------------------------
% END FRONT MATTER -------------------------------------------------------
%------------------------------------------------------

%------------------------------------------------------
% BEGIN MAIN MATTER -------------------------------------------------------
%------------------------------------------------------

%------------------------------------------------------
% 		SECTION: INTRODUCTION
%------------------------------------------------------

\section{Introduction}
In Riemannian geometry, collapse imposes strong geometric and topological restrictions on the spaces on which it occurs. Recall that that there are eight  three-dimensional \emph{Thurston geometries}: $\mathbb{S}^3$, $\mathbb{R}^3$, $\mathbb{H}^3$, $\mathbb{S}^2\times \mathbb{R}$, $\mathbb{H}^2\times \mathbb{R}$, $\SL$, $\Nil$ and $\Sol$ (see \cite{Sco}). A closed (i.e. compact and without boundary) $3$-manifold is \emph{geometric} if it admits a geometric structure modeled on one of these geometries.  In this context, Shioya and Yamaguchi \cite{ShiYam} obtained the following result for sufficiently collapsed Riemannian $3$-manifolds.

% THM: SHIOYA-YAMAGUCHI 3D COLLAPSE

\begin{thm*}[Shioya and Yamaguchi \protect{\cite[Corollary 0.9]{ShiYam}}]
\label{T:SHIOYA_YAMAGUCHI}
Let $D>0$ and let $\mathcal{M}(3,D)$ be the class of closed Riemannian $3$-manifolds with diameter at most $D$. For any $D>0$, there exists a constant $\varepsilon=\varepsilon(D)>0$ such that if a closed, prime $3$-manifold with infinite fundamental group admits a Riemannian metric contained in $\mathcal{M}(3,D)$ with volume $<\varepsilon$, then it admits a geometric structure modeled on one of the six geometries $\mathbb{R}^3$, $\mathbb{S}^2\times \mathbb{R}$, $\mathbb{H}^2\times \mathbb{R}$, $\SL$, $\Nil$ and $\Sol$.
\end{thm*}

Hyperbolic geometry is ruled out by the fact that a manifold with such a geometry has non-vanishing simplicial volume (see \cite[Theorem 6.2]{Thurston1978}), and the fact that manifolds with positive simplicial volume do not collapse (see \cite[pp.~67--68]{Gromov01}). We point out that, although the geometry $\mathbb{S}^3$ appears in the original statement of the result in \cite{ShiYam}, the assumption that the fundamental group of the manifold is infinite precludes it from admitting spherical geometry. 
Note, in any case, that by Perelman's proof of Thurston's Elliptization Conjecture, any closed $3$-dimensional manifold with finite fundamental group has spherical geometry.

 Closed Riemannian manifolds are special cases of Alexandrov spaces (with curvature bounded below). In this article, we extend the preceding theorem to the Alexandrov setting.  In dimension three, the structure of Alexandrov spaces is 
fairly well understood (see \cite{GalGui,MitYam}). Topologically, these spaces are either $3$-manifolds or quotients of $3$-manifolds by orientation reversing involutions with only isolated fixed points (see \cite{GalGui}). 

Before stating our main theorem, we need some definitions. Recall that a non-trivial closed $3$-manifold $M$ is \emph{prime} if it cannot be presented as a connected sum  of two  non-trivial closed $3$-manifolds. A closed $3$-manifold is \emph{irreducible} if every embedded $2$-sphere bounds a $3$-ball. It is known that, with the exception of manifolds homeomorphic to $\mathbb{S}^3$ or $\mathbb{S}^1\times \mathbb{S}^2$, a closed $3$-manifold is prime if and only if it is irreducible (see \cite[Lemma~1]{Mi}). Since $\mathbb{S}^1\times \mathbb{S}^2$ is geometric, and $\mathbb{S}^3$ has finite fundamental group, one can think of the theorem above as a statement about irreducible $3$-manifolds. Therefore, in generalizing this theorem to Alexandrov spaces we will focus our attention on the irreducible case. In view of this, we need to define \emph{irreducibility} for this more general class of spaces.

% DEF: IRREDUCIBILITY

\begin{definition*} Let $X$ be a closed Alexandrov $3$-space. We say that  $X$ is \emph{irreducible} if every embedded $2$-sphere in $X$ bounds a $3$-ball and, in the case that the set of topologically singular points of $X$ is non-empty, it is further required that every $2$-sided $\mathbb{R}P^2$ bound a $K(\mathbb{R}P^2)$, a cone over $\mathbb{R}P^2$. 
\end{definition*}

With this definition in hand, we state our main result. 

% THEOREM: EXTENSION OF SHIOYA-YAMAGUCHI'S RESULT

\begin{theorem}
\label{THM:MAIN_THEOREM}
For any $D>0$ there exists $\varepsilon=\varepsilon(D)>0$ such that, if $X$ is a closed, irreducible Alexandrov $3$-space of $\mathrm{curv}\geq -1$ satisfying that $\diam X\leq D$ and $\vol X <\varepsilon$, then  $X$ admits a geometric structure modeled on one of the seven geometries $\mathbb{R}^3$, $\mathbb{S}^3$, $\mathbb{S}^2\times \mathbb{R}$, $\mathbb{H}^2\times \mathbb{R}$, $\SL$, $\Nil$ and $\Sol$.
\end{theorem}

 In Theorem~\ref{THM:MAIN_THEOREM}, by the \emph{volume} of an Alexandrov $3$-space, we mean its $3$-dimensional Hausdorff measure, normalized so that volume of $3$-dimensional Riemannian manifolds agrees with the usual Riemannian volume.
 
 % REMARK

\begin{remark}
\label{R:HYP_GEOM_NOT}
As in the Riemannian case, one can rule out the appearance of hyperbolic geometry by combining the fact that the simplicial volume of a collapsing Alexandrov space is zero (see \cite[Corollary 1.7]{MitYam2}) with the fact that the simplicial volume of a hyperbolic manifold must be bounded below by the Riemannian volume (see \cite[Theorem 6.2]{Thurston1978}). For more details, see the end of Section~\ref{S:OUTLINE_OF_PROOF}.
\end{remark}
%While in the Riemannian case one is able to rule out the appearance of hyperbolic geometry, in the Alexandrov case we are not able to do so, since it is not known whether or not positive simplicial volume poses an obstruction to Alexandrov collapse. 

We prove Theorem~\ref{THM:MAIN_THEOREM} by carefully studying the metric and topological structure of collapsed irreducible  Alexandrov $3$-spaces and their orientable double branched covers (in the case where the space is not a manifold).  The extensive work of Mitsuishi and Yamaguchi on collapsed Alexandrov $3$-spaces \cite{MitYam}, combined with the irreducibility hypothesis, allows us to obtain fairly explicit topological descriptions of closed, collapsed, irreducible Alexandrov $3$-spaces, and to exhibit them as geometric $3$-manifolds or their quotients by orientation-reversing involutions with only isolated fixed points. Classification results of Alexandrov spaces with (local) circle actions (see \cite{GalNun,NZ}) and classical results on involutions on $3$-manifolds (see \cite{KimToll,Kwun}) also play an important role in the proof. 
\\

Our article is divided as follows. In Section~\ref{S:PRELIM} we collect results on general Alexandrov $3$-spaces, as well as on those spaces that admit collapse or (local) isometric circle actions. In Section~\ref{S:DBL_BRANCHED_COLAPSE} we show that a non-manifold Alexandrov $3$-space collapses if and only if its canonical orientable double branched Alexandrov cover also collapses. We prove Theorem~\ref{THM:MAIN_THEOREM} in Sections~\ref{S:OUTLINE_OF_PROOF}--\ref{S:ZERO_DIM_LIMIT}. 

% ACKNOWLEDGEMENTS

\begin{ack}The authors thank John Harvey for pointing out reference \cite{MitYam2}.
\end{ack}

%------------------------------------------------------
% 		SECTION: PRELIMINARIES
%------------------------------------------------------

\section{Preliminaries}
\label{S:PRELIM}

% SUBSECTION: 3D ALEXANDROV SPACES

\subsection*{Three-dimensional Alexandrov spaces}

In this section we give a brief account of the basic structural properties of closed Alexandrov spaces of dimension $3$. We refer the reader to \cite{GalGui} for a more detailed account. 

Let $X$ be a closed Alexandrov $3$-space. The space of directions $\Sigma_x X$ at each point $x$ in $X$ is a closed Alexandrov $2$-space with curvature bounded below by $1$. Therefore, by the Bonnet-Myers Theorem \cite[Theorem 10.4.1]{BBI}, the fundamental group of $\Sigma_x X$ is finite. This in turn implies that $\Sigma_x X$ is homeomorphic to a $2$-sphere $\mathbb{S}^2$ or a real projective plane $\mathbb{R}P^2$. A point in $X$ whose space of directions is homeomorphic to $\mathbb{S}^2$ is called a \textit{topologically regular point}. Otherwise, the point is called \textit{topologically singular}. The set of topologically regular points is open and dense in $X$.

By Perelman's Conical Neighborhood Theorem \cite{Per2}, every point $x$ in $X$ has a neighborhood pointed-homeomorphic to the cone over the $\Sigma_xX$. This result implies that the set of topologically singular points of $X$ is finite and $X$ is homeomorphic to a compact $3$-manifold with a finite number of $\mathbb{R}P^2$-boundary components to which one glues in cones over $\mathbb{R}P^2$. It is an easy consequence of this description that $X$ must have an even number of topologically singular points.

A closed Alexandrov $3$-space $X$ whose set of topologically singular points is non-empty can also be described as a
quotient of a closed, orientable, topological $3$-manifold $\tilde{X}$ by an orientation-reversing involution $\iota: \tilde{X}\to \tilde{X}$ with only fixed points. The $3$-manifold $\tilde{X}$ is the so-called \textit{orientable double branched cover of $X$} (see, for example, \cite[Lemma 1.7]{GalGui}). In addition to this topological description, it is possible to lift the metric on $X$ to $\tilde{X}$, so that $\tilde{X}$ becomes an Alexandrov space with the same lower curvature bound as $X$ and $\iota$ is an isometry with respect to the lifted metric. In particular, $\iota$ is equivalent to a smooth involution on $\tilde{X}$ regarded as a smooth $3$-manifold. We refer the reader to \cite[ Lemma 1.8]{GalGui} and \cite[Section 5]{GW} for more details.

The geometric structure of Alexandrov $3$-spaces was further studied in \cite{GalGui}. We recall that $X$ is \textit{geometric} (or admits a Thurston geometry) if it is a quotient of one of the eight $3$-dimensional Thurston geometries by some cocompact lattice (see \cite{Sco}). We say that $X$ \textit{admits a geometric decomposition} if $X$ can be cut along a family of spheres, projective planes, tori, and Klein bottles in such a way that the resulting pieces are geometric. The \textit{geometrization of closed Alexandrov $3$-spaces} was obtained in \cite{GalGui}, i.e., it was proved that every closed Alexandrov $3$-space admits a geometric decomposition into geometric Alexandrov $3$-spaces.

%SUBSECTION: COLLAPSING ALEXANDROV 3-SPACES 
%           (MITSUISHI-YAMAGUCHI)

\subsection*{Collapsing Alexandrov 3-spaces} Let $\{X_i\}_{i=1}^{\infty}$ be a sequence of $n$-di\-men\-sion\-al Alexandrov spaces with diameters uniformly bounded above by $D>0$ and $\curv\geq k$ for some $k\in \RR$.  After passing to a subsequence, Gromov's Precompactness Theorem implies that there exists an Alexandrov space $Y$ with $\diam Y\leq D$ and $\curv Y\geq k$ such that $X_i\GHto Y$. As in the Riemannian case, the sequence $X_i$ is said to \textit{collapse} to $Y$ if $\dim Y<n$. We will also say that an $n$-dimensional Alexandrov space $X$ \textit{collapses} (or that it is a \textit{collapsing Alexandrov space}) if there exists a sequence of Alexandrov metrics $\{d_i\}_{i=1}^{\infty}$ on $X$, such that $\{(X,d_i)\}_{i=1}^{\infty}$ is a collapsing sequence. In \cite{MitYam}, A.\ Mitsuishi and T.\ Yamaguchi obtained a topological classification of closed collapsing $3$-dimensional Alexandrov spaces, describing them as a union of certain pieces. We now give a brief account of those pieces with topological or metric singularities.
\\  

\noindent\emph{The space $\bpt$.} Let $D^2\times \mathbb{S}^1\subset \RR^2\times \CC$ be equipped with the usual flat product metric. An isometric involution $\alpha$ on $D^2\times \mathbb{S}^1$ is defined by 
\[
\alpha(x, e^{i\theta}):= (-x, e^{-i\theta}).
\]
The space $\bpt:= D^2\times \mathbb{S}^1/\alpha $ is an Alexandrov space of $\curv\geq 0$ with two topologically singular points. There is a natural projection $p\colon \bpt\to K_1(\mathbb{S}^1)$ sending an interval joining the topologically singular points to the vertex $o$ of the cone. This projection is a fibration on $K_1(\mathbb{S}^1)\setminus\{o\}$. 

{ 
An alternate, topological description of $\bpt$ appears after \cite[Example 2.60]{MitYam}: choose two copies of cones over $\mathbb{R}P^2$, select a  disc $D^2_i$, $i=0,1$ on each $\mathbb{R}P^2$-boundary, and glue both cones using some homeomorphism $\varphi:D^2_0\to D^2_1$. The resulting space does not depend on the chosen $\varphi$, and is homeomorphic to $\bpt$. It is clear that its boundary is obtained by taking two M\"{o}bius bands glued by their boundaries, i.e, a Klein bottle.
}
\\

% SPACES WITH 2-DIMENSIONAL SOULS

\noindent\emph{Spaces with $2$-dimensional souls.} We now describe three different closed Alexandrov $3$-spaces as quotients of certain involutions (cf. \cite{Nat})
\begin{itemize}
\item[(i)] $\bstwo:= \mathbb{S}^2\times [-1,1]/(\sigma,-\mathrm{id})$, where $\mathbb{S}^2$ is a sphere of non-negative curvature in the Alexandrov sense with an isometric involution $\sigma:\mathbb{S}^2\to \mathbb{S}^2$ topologically conjugate to the involution on the $2$-sphere given by the suspension of the antipodal map on the circle.
The resulting space is a topological manifold, in spite of having two points whose space of directions is not isometric to a round sphere.
\vspace{.2cm}
\item[(ii)]  $\bsfour:= T^2\times [-1,1]/(\sigma,-\mathrm{id})$, where $T^2$ is a flat torus and $\sigma:T^2\rightarrow T^2$ is an isometric involution on $T^2$ topologically conjugate to the unique involution on $T^2$ whose quotient is $\mathbb{S}^2$. 
{ 
This space has four topologically singular points, corresponding to the four fixed points of the involution; this can be easily seen by observing that at each such point, the differential of the involution acts as the antipodal map on the unit tangent sphere. Its oriented branched cover is clearly $\mathbb{T}^2\times [-1,1]$.
} 

\vspace{.2cm}
\item[(iii)] $\bptwo:= K^2\times [-1,1]/(\sigma,-\mathrm{id})$, where 
$K^2$ is a flat Klein bottle and $\sigma:K^2\to K^2$ is an isometric involution topologically conjugate to the unique involution on  $K^2$ whose quotient is $\RR P^2$.

\end{itemize}	

\noindent\emph{Generalized Seifert fiber spaces.} A \textit{generalized Seifert fibration} of a topological $3$-orbifold $M$ over a topological $2$-orbifold $B$ (both possibly with boundaries) is a map $f\colon M\to X$ whose fibers are homeomorphic to circles or bounded closed intervals. It is required that for every $x\in B$, there is a neighborhood $U_x$ homeomorphic to a $2$-disk such that
\begin{itemize}
\item[(i)] if $f^{-1}(x)$ is homeomorphic to a circle, then there exists a fiber-preserving homeomorphism of $f^{-1}(U_x)$ to a Seifert fibered solid torus in the usual sense, and
\item[(ii)] if $f^{-1}(x)$ is homeomorphic to an interval, then there is a fiber-preserving homeomorphism of $f^{-1}(U_x)$ to the space $\bpt$, with respect to the fibration $(\bpt,p^{-1}(o))\to (K_1(\mathbb{S}^1),o)$. 
\end{itemize}
Furthermore, for any compact component $C$ of $\partial B$ there is a collar neighborhood $N$ of $C$ in $B$ such that $f|_{f^{-1}(N)}$ is a usual circle bundle over $N$. We say that $M$ is a \textit{generalized Seifert fibered space} and we use the notation $M=\seif(B)$. 
\\

\noindent\emph{Generalized solid tori and Klein bottles.} A \emph{generalized solid torus} (GST) (respectively, \emph{generalized solid Klein Bottle} (GSKB)) is a topological $3$-orbifold $Y$ with boundary homeomorphic to a torus (respectively, a Klein bottle). It admits a map $Y\to \mathbb{S}^1$ such that the fibers are homeomorphic to either a $2$-disk or a M\"{o}bius band, and the  fiber type can only change at a finite number of \textit{corner points} in $\mathbb{S}^1$. We refer the reader to \cite[Definition 1.4]{MitYam} for the precise definitions. A GST (respectively, a GSKB) is said to be \emph{of type $N$} if it has $2N$ topologically singular points. The GST of type $0$ are defined to be  $\mathbb{S}^1\times D^2$ and $\mathbb{S}^1\times \mo$, where $\mo$ denotes the M\"obius strip. Similarly, one defines the GSKB of type $0$ as $\mathbb{S}^1\widetilde{\times} D^2$ and $\mathbb{S}^1\tilde{\times}\mo$. For convenience, we will denote a generalized solid torus (respectively,  generalized solid Klein bottle) of type $N$ by $GST_N$ (respectively, $GSKB_N$).  
\\

\noindent\emph{$I$-bundles over the Klein bottle.}  These are obtained as disk bundles of certain line bundles over the Klein bottle $K^2$. They are easily described as quotients of $\mathbb{R}^3$ under certain isometric  actions \cite{Wo}. Except for the trivial bundle $K^2\times I$, the rest are as follows:
\begin{enumerate}
\item[(i)] $K^2\widetilde{\times} I$: this is the disc bundle in the orientable three manifold obtained as quotient of $\mathbb{R}^3$ under the group generated by
\[
(x,y,z)\stackrel{\tilde{\tau}}{\to} (x+1,y,z), \qquad (x,y,z)\stackrel{\tilde{\sigma}}{\to} (-x, y+1,-z).
\]
Its boundary is given by a $2$-torus.
\item[(ii)] $K^2\hat{\times} I$: this is the disc bundle in the non-orientable three manifold obtained as quotient of $\mathbb{R}^3$ under the group generated by
\[
(x,y,z)\stackrel{\hat{\tau}}{\to} (x+1,y,-z), \qquad (x,y,z)\stackrel{\hat{\sigma}}{\to} (-x, y+1,-z).
\]
Its boundary is given by a Klein bottle.
\end{enumerate}

Observe that the correspondence sending $\tilde{\tau}\to\hat{\tau}^2$, 
$\tilde{\sigma}\to\hat{\sigma}$ induces an injective homomorphism from the fundamental group of $K^2\widetilde{\times} I$ into that of $K^2\hat{\times} I$, thus showing that $K^2\widetilde{\times} I$ is a twofold cover of $K^2\hat{\times} I$. Furthermore, since the 
fundamental group of $K^2$ is the dihedral group, and this group contains a unique subgroup of index 2, it follows that $K^2\widetilde{\times} I$ is the unique twofold cover of 
$K^2\hat{\times}I$.

We summarize the results of \cite{MitYam} in Tables \ref{TBL:2-DIM_LIMIT}, \ref{TBL:1-DIM_LIMIT} and \ref{TBL:0-DIM_LIMIT}, which contain the classifications of collapsing Alexandrov $3$-spaces with a $2$-, $1$- and $0$-dimensional limit space, respectively.

%-----------------------------------------------------------------------------------------
% TABLE: COLLAPSING ALEXANDROV 3-SPACES 2-DIM. LIMIT
%-----------------------------------------------------------------------------------------

\begin{table}[h]
\centering
\caption{Two-dimensional limit}
\label{TBL:2-DIM_LIMIT}
\bgroup
\def\arraystretch{1.5}
\begin{tabular}{|c|c|c|}
\hline
\rowcolor[HTML]{EFEFEF} 
$\dim Y$              &\hspace{0.04cm} $\partial Y$     & Homeomorphism type of $X_i$                                                                                                     \\[.1cm] \hline
                      & $=\emptyset$     & $\seif(Y)$                                                                                                                      \\[.1cm] \cline{2-3} 
\multirow{-2}{*}{$2$} & $\neq \emptyset$ & $\seif(Y)\bigcup_{\partial Y}\left(\left(\bigsqcup_i GST_{N_i}\right) \bigsqcup \left( \bigsqcup_j GSKB_{N_j} \right) \right) $\\[.1cm]
\hline
\end{tabular}
\egroup
\end{table}

%-----------------------------------------------------------------------------------------
% TABLE: COLLAPSING ALEXANDROV 3-SPACES 1-DIM. LIMIT
%-----------------------------------------------------------------------------------------

\begin{table}[h]
\centering
\caption{One-dimensional limit}
\label{TBL:1-DIM_LIMIT}
\bgroup
\def\arraystretch{1.5}
\begin{tabular}{|c|c|c|c|c|}
\hline
\rowcolor[HTML]{EFEFEF} 
\hspace{0.1cm}$\dim Y$               & $\partial Y$                        & \multicolumn{2}{c}{\cellcolor[HTML]{EFEFEF} \hspace{1.45cm} Homeomorphism type of $X_i$}                                                                  & \\ \hline
                       & $=\emptyset$                        & \multicolumn{3}{c|}{$F$-fiber bundle over $S^1$ with $F=\mathbb{S}^2$, $\mathbb{R}P^2$, $T^2$, $K^2$                     }                          \\ \cline{2-5} 
                       &                                     &                                               &     $\partial B$, $\partial B'$ & $B$, $B'$         \\   \cline{4-5} 
                       % \cellcolor[HTML]{EFEFEF}$\partial B$, $\partial B'$ & \cellcolor[HTML]{EFEFEF}$B$, $B'$         \\   \cline{4-5} 
                       &                                     &                                               &                                                     & $D^3$,                                    \\ \cline{5-5} 
                       &                                     &                                               &                                                     & $\mathbb{R}P^3\setminus \mathrm{int}D^3$, \\ \cline{5-5} 
                       &                                     &                                               & \multirow{-3}{*}{$\mathbb{S}^2$}                    & $\bstwo$                                  \\ \cline{4-5} 
                       &                                     &                                               & $\mathbb{R}P^2$                                     & $K_1(\mathbb{R}P^2)$                      \\ \cline{4-5} 
                       &                                     &                                               &                                                     & $D^2\times S^1$,                          \\ \cline{5-5} 
                       &                                     &                                               &                                                     & $\mathrm{Mo}\times S^1$,                  \\ \cline{5-5} 
                       &                                     &                                               &                                                     & $K^2\tilde{\times} I$,                    \\ \cline{5-5} 
                       &                                     &                                               & \multirow{-4}{*}{$T^2$}                             & $\bsfour$                                 \\ \cline{4-5} 
                       &                                     &                                               &                                                     & $\mathbb{S}^1\tilde{\times} D^2$,                  \\ \cline{5-5} 
                       &                                     &                                               &                                                     & $K^2\hat{\times} I$,                      \\ \cline{5-5} 
                       &                                     &                                               &                                                     & $\bpt$,                                   \\ \cline{5-5} 
\multirow{-14}{*}{$1$} & \multirow{-13}{*}{$\neq \emptyset$} & \multirow{-13}{*}{$B\bigcup_{\partial B} B'$} & \multirow{-4}{*}{$K^2$}                             & $\bptwo$                                 \\
\hline
\end{tabular}
\egroup
\end{table}

%-----------------------------------------------------------------------------------------
% TABLE: COLLAPSING ALEXANDROV 3-SPACES 0-DIM. LIMIT
%-----------------------------------------------------------------------------------------

\begin{table}[h]
\centering
\caption{Zero-dimensional space}
\label{TBL:0-DIM_LIMIT}
\bgroup
\def\arraystretch{1.5}
\begin{tabular}{|c|c|}
\hline
\rowcolor[HTML]{EFEFEF} 
$\dim Y$              & Homeomorphism type of $X_i$                                                           \\ \hline
                      & That of a space in Table \ref{TBL:2-DIM_LIMIT} with $\curv Y\geq 0$                \\ \cline{2-2} 
                      & That of a space in Table \ref{TBL:1-DIM_LIMIT}                                     \\ \cline{2-2} 
\multirow{-3}{*}{$0$} & That of a closed non-negatively curved Alexandrov space\\[-1.2ex]
& with finite fundamental group\\
\hline
\end{tabular}
\egroup
\end{table}

%SUBSECTION: LOCAL CIRCLE ACTIONS

\subsection*{Local circle actions}

In this section we summarize the classification of closed Alexandrov spaces of dimension $3$ admitting a local isometric circle action obtained in \cite{GalNun}. This classification will be useful in subsequent sections.

Let $X$ be a closed Alexandrov $3$-space. A \textit{local circle action on $X$} is a decomposition of $X$ into (possibly
degenerate) disjoint, simple, closed curves called \textit{fibers}, each having a tubular neighborhood which admits an effective circle action whose orbits are the curves of the decomposition. The local circle action is \textit{isometric} if the circle actions on each tubular neighborhood of the fibers are isometric with respect to the restricted metric.

The different  fiber types of a local circle action on $X$ depend on the corresponding isotropy group, where one considers these fibers as orbits of the isometric circle action on a small tubular neighborhood around them. The possible fiber types are the following.  \emph{$F$-fibers} are topologically regular fixed-point fibers, while \emph{$SF$-fibers} are topologically singular fixed-point fibers. The fibers of type $E$ are those that correspond to $\mathbb{Z}_k$ isotropy, acting in such a way that local orientation is preserved. The fibers of type $SE$ correspond to $\mathbb{Z}_2$ isotropy, reversing the local orientation. Fibers that are not $F$-, $SF$-, $E$- or $SE$-fibers are called $R$-fibers. The strata of $F$-, $SF$-, $E$-, $SE$- and $R$-fibers are denoted, respectively, by $F$, $SF$, $E$, $SE$ and $R$. The fiber space, denoted by $X^*$, is a $2$-dimensional Alexandrov space, and therefore a topological $2$-manifold. Its boundary is composed of the images of $F$-, $SF$- and $SE$-fibers under the fiber projection map, while the interior of $X^*$ consists of $R$-fibers and a finite number of $E$-fibers.

A closed Alexandrov $3$-space with an isometric local circle action can be decomposed into \textit{building blocks} of types $F$, $SF$, $E$ and $SE$, which arise when considering small tubular neighborhoods of connected components of fibers of the corresponding type. A building block is called \textit{simple} if its boundary is homeomorphic to a torus, and \textit{twisted} if its boundary is homeomorphic to a Klein bottle (see \cite[Section 3]{GalNun}). The stratum of $R$-fibers is an $S^1$-fiber bundle with structure group $\mathrm{O}(2)$.

The set of isometric local circle actions is in one-to-one correspondence with the following set of equivariant invariants of the fiber space $X^*$:
\begin{equation*}
\left\{b; \varepsilon, g, (f,k_1), (t,k_2), (s,k_3); \{ (\alpha_i, \beta_i) \}_{i=1}^n; (r_1,r_2, \ldots, r_{s-k_{3}}); (q_1, q_2, \ldots, q_{k_3})\right\}.
\end{equation*}

The definition of these invariants is as follows. We let $(\varepsilon,k)$ be the pair that classifies the $S^1$-bundle of $R$-fibers according to \cite[Theorem 3.2]{GalNun}. We denote the genus of $X^{*}$ by $g\geq 0$.  The symbols $f, t, k_1, k_2$ are non-negative integers  such that $k_1 \leq f$ and $k_2\leq t$, where $k_1$ is the number of twisted $F$-blocks and $k_2$ is the number of twisted $SE$-blocks. Consequently $f-k_1$ is the number of simple $F$-blocks and $t-k_2$ is the number of simple $SE$-blocks. The number of $E$-fibers will be denoted by $n$ and we let $\{ (\alpha_i, \beta_i)\}_{i=1}^n$ be the corresponding Seifert invariants (see \cite[Section 1.7]{O}). We also let $b$ be an integer or an integer mod $2$ with the following conditions: $b=0$ if $f+t>0$ or if $\varepsilon\in\{o_2,n_1,n_3,n_4\}$ and some $\alpha_i=2$ (see \cite[Theorem 3.2]{GalNun} for the precise definitions of the $o_i$, $n_j$ and $\alpha_l$); $b\in\{0,1\}$ if $f+t=0$ and $\varepsilon\in\{ o_2,n_1,n_3,n_4\}$ and all $\alpha_i\neq 2$. In the remaining cases $b$ is an arbitrary integer. 
We let $s, k_3$ be non-negative integers, where  $k_3\leq s$ is the number of twisted $SF$-blocks. Hence $s-k_3$ is the number of simple $SF$-blocks, and we let $(r_1, r_2, \ldots, r_{s-k_{3}})$ and $(q_1, q_2, \ldots, q_{k_3} )$ be $(s-k_{3})$- and $k_3$-tuples of non-negative even integers corresponding to the number of topologically singular points in each  simple and twisted $SF$-block, respectively. The numbers $k$, $k_1$, $k_2$, and $k_3$ satisfy $k_1 + k_2 + k_3 = k$.

A closed Alexandrov space admitting an isometric local circle action given by the previous set of invariants can be decomposed as an equivariant connected sum. 

\begin{thm}[Theorem B \cite{GalNun} ]
\label{THM:INVARIANTS}
Let $X$ be a closed, connected Alexandrov $3$-space with $2r\geq 0$ topologically singular points.  Assume that $X$ admits the isometric local circle action given by the set of invariants
\[
\left\{b; \varepsilon, g, (f,k_1), (t,k_2), (s,k_3); \{ (\alpha_i, \beta_i) \}_{i=1}^n; (r_1,r_2, \ldots, r_{s-k_{3}}); (q_1, q_2, \ldots, q_{k_3})\right\}.
\]
Then $X$ is equivariantly homeomorphic to 
\[
M\#\underbrace{\Susp(\Real P^2)\# \cdots \# \Susp(\Real P^2)}_{r \text{ summands}},
\]
where $M$ is the closed $3$-manifold determined by the set of invariants 
\[
\left\{ b; \varepsilon, g, (f+s,k_1+k_3), (t,k_2); \{ (\alpha_i, \beta_i) \}_{i=1}^n \right\}
\]
\end{thm}

%---------------------------------------------------------------------------------------
% SECTION: DOUBLE BRANCHED COVERS OF COLLAPSING 
%          ALEXANDROV SPACES
%----------------------------------------------------------------------------------------

% NOTES: 

\section{Double branched covers of collapsing Alexandrov spaces}
\label{S:DBL_BRANCHED_COLAPSE}

Let $X$ be a finite-dimensional, compact, non-orientable Alexandrov $3$-space, $\mu$ its $3$-dimensional Hausdorff measure and $S_X$ its set of topologically singular points. We let $\tilde{X}$ be the  orientable double branched cover of $X$ equipped with $\tilde{\mu}$, its $3$-dimensional Hausdorff measure. Let $\pi\colon \tilde{X}\to X$ be the canonical projection and $\pi_*\tilde{\mu}$, the push-forward measure of $\mu$ with respect to $\pi$. 

% COROLLARY

\begin{prop}
\label{C:COLLAPSE_COVER}
Let $X$ be a non-orientable Alexandrov $3$-space. Then  $\pi_*\tilde{\mu}= 2\mu$. 

\end{prop}

% PROOF

\begin{proof}
Let $X'=X\setminus S_X$ and observe that, since $S_X$ is a discrete subset of $X$, its $\mu$-measure is zero.
Therefore, it suffices to show that $\pi_*\tilde{\mu}(B)= 2\mu(B)$ for sufficiently small balls centered at points in $X'$. Observe that we may assume, without loss of generality, that each one of these balls is evenly covered. We have  that $\pi^{-1}(B)\subset \tilde{X}'$ is the disjoint union of two balls $B_1$, $B_2$ in $\tilde{X}'$ such that $\pi|_{B_i}\colon B_i\to B$ is an isometry. Therefore, $\tilde{\mu}(B_i)=\mu(B)$. The proposition now follows from the definition of the push-forward measure.
\end{proof}

% COROLLARY

\begin{cor}
\label{COR:COLLAPSE}
Let $\{X_i\}$ be a sequence of compact non-manifold Alexandrov $3$-spaces. Then the following hold:
\begin{itemize}
	\item[(i)]  The sequence $\{X_i\}$ converges (in the Gromov-Hausdorff metric) to an Alexandrov space $X_\infty$ if and only if the sequence of orientable branched double covers $\{\tilde{X}_i\}$, equipped with their canonical involutions $\sigma_i$, converges in the equivariant Gromov-Hausdorff topology to an Alexandrov space $\tilde{X}_\infty$ equipped with a limit involution $\sigma_\infty$ so that $X_\infty=\tilde{X}_\infty/\sigma_\infty$. 
	\item[(ii)] The sequence $\{X_i\}$ collapses if and only if the sequence of orientable branched double covers $\{\tilde{X}_i\}$ collapses.
\end{itemize} 
\end{cor}

% PROOF
\begin{proof}
We first prove part (i). 
Let $\{(X_i, d_i)\}$ be a convergent sequence of compact non-manifold Alexandrov $3$-spaces. For simplicity, we let $X_i=(X_i,d_i)$ and we let $X_\infty$ be the limit of the sequence. We have an induced sequence of  metrics $\tilde{d}_i$ on $\tilde{X}_i$. Let $\alpha_i\colon\tilde{X}_i\rightarrow \tilde{X}_i$ be the isometric involution so that $\tilde{X}_i/\alpha_i$ is isometric to $X_i$. Observe that the lift of any $\varepsilon$-net in $X_i$ is an $\varepsilon$-net in $\tilde{X}_i$ that is invariant under the involution $\sigma_i$. This implies that the sequence $\{\tilde{X}_i\}$ has a limit, which we denote by $\tilde{X}_\infty$.
We define $\alpha_{\infty}\colon\tilde{X}_{\infty}\to\tilde{X}_{\infty}$ as the limit of the $\alpha_i$. Now, by \cite[Theorem 2-1]{F}, for every $i$, we have

\begin{equation}
\label{equivariant_vs_normal}
d_{GH}(\tilde{X}_i/\alpha_i, \tilde{X}_{\infty}/\alpha_{\infty})\leq C\left( d_{eGH} ((\tilde{X}_i, \mathbb{Z}_2),(\tilde{X}_{\infty},\mathbb{Z}_2)) \right)^{1/3},
\end{equation}
where $d_{eGH}$ stands for the equivariant Gromov-Hausdorff distance. 
By \cite[Proposition 3-6]{FY}, $\tilde{X}_i \GHto \tilde{X}_{\infty}$ implies that there is a subsequence $\{\tilde{X}_{i_k}\}$ converging in the equivariant Gromov-Hausdorff topology. Furthermore, by the proof of the same proposition, the group acting on the limit is the inverse limit of the groups acting on the members of the subsequence, which in this case are all $\mathbb{Z}_2$. Therefore, $(\tilde{X}_{i_k},\mathbb{Z}_2)\eGHto(\tilde{X}_{\infty},\mathbb{Z}_2)$. 

Now, inequality (\ref{equivariant_vs_normal}) implies that $\tilde{X}/\alpha_i\GHto\tilde{X}_{\infty}/\alpha_{\infty}$. However, since $\tilde{X}/\alpha_i$ is isometric to $X_i$, we have that $X_i\GHto\tilde{X}_{\infty}/\alpha_{\infty}$. Since we already had that $X_i\GHto X_{\infty}$, we conclude that $X_\infty$ is isometric to $\tilde{X}_{\infty}/\alpha_{\infty}$. 
Observe that the action of $\alpha_{\infty}$ on $\tilde{X}_{\infty}$ might be trivial. This shows that, if $\{X_i\}$ GH-converges to $X_\infty$, then $\{\tilde{X}\}_i$ converges in the equivariant Gromov-Hausdorff distance to $\tilde{X}_\infty$ with an involution $\sigma_\infty$ so that $\tilde{X}_\infty/\sigma_\infty$ is isometric to $X_\infty$. 

On the other hand, equation \eqref{equivariant_vs_normal} implies that if $\{\tilde{X}_i\}$, equipped with their canonical involutions $\sigma_i$ converges in the equivariant Gromov-Hausdorff distance to $\tilde{X}_\infty$ with a limit involution $\sigma_\infty$, then $\{X_i\}$ converges to $X_\infty$ in the Gromov-Hausdorff distance. This shows part (i) of the corollary.
Part (ii) follows from part (i) and Proposition~\ref{C:COLLAPSE_COVER}. 
\end{proof}

% COROLLARY

\begin{cor}
Let $M$ be a closed Riemannian manifold with positive minimal volume admitting an isometric involution $\iota:M\to M$ with a discrete set of fixed points. Then the Alexandrov space $M/\iota$ has positive minimal volume. 
\end{cor}

%------------------------------------------------------
%	SECTION: OUTLINE OF PROOF OF THEOREM 
%------------------------------------------------------
\section{Outline of the proof of Theorem \ref{THM:MAIN_THEOREM}}
\label{S:OUTLINE_OF_PROOF}

We proceed by contradiction. Suppose the result does not hold. Then we have a sequence of closed, irreducible Alexandrov $3$-spaces $\{X_i\}_{i=1}^{\infty}$ with $\curv X_i\geq -1$ and $\diam X_i\leq D$ such that $\vol X_i \to 0$ and $X_i$ is not geometric. Therefore, passing to a subsequence if necessary, we can assume that $\{X_i\}$ collapses to a compact Alexandrov space $Y$. We will separate our analysis in three cases according to the dimension of $Y$. In Sections \ref{S:TWO_DIM_LIMIT}, \ref{S:ONE_DIM_LIMIT} and \ref{S:ZERO_DIM_LIMIT} we will address the cases in which $Y$ has dimension $2$, $1$ and $0$ respectively. In every case, we conclude that $X_i$ is geometric for a sufficiently big $i$, obtaining a contradiction.

 The following Lemma will be useful in our analysis. We refer the reader to \cite{Thur} for the definitions and basic properties of orbifold-covering spaces.

% LEMMA 

\begin{lem} Let $X$ be an Alexandrov $3$-space which is not a manifold. Then $X$ is geometric if and only if its orientable branched double cover $\tilde{X}$ is geometric.
\end{lem}

% PROOF

\begin{proof}
We assume first that $X$ is geometric. Then, there exists a Thurston geometry $M^3$ such that $X$ is isometric to $M^3/G$, where $G$ is a co-compact lattice. In particular, $M^3$ is an orbifold-covering space of $X$ and $X$ is a good orbifold. Therefore, $X$ admits a Riemannian orbifold metric and we will henceforth assume that it has such a metric. Moreover, $M^3$ is the universal orbifold-covering space of $X$. 

Now we note that the canonical projection $\pi\colon\tilde{X}\rightarrow X$ is an orbifold-covering map. Then, by the universality of $M^3$, there exists an orbifold-covering map $M^3\rightarrow \tilde{X}$. Therefore, there exists a subgroup $\tilde{G}$ of the group of deck transformations  such that $\tilde{X}$ is isometric to $M^3/\tilde{G}$. Furthermore, the metric on $M^3$ is the lift of the metric on $X$ and it coincides with the lift of the metric on $\tilde{X}$. Therefore, $\tilde{X}$ is geometric. 

Assume now that $\tilde{X}$ is geometric, i.e.\ it is isometric to $M^3/\tilde{G}$, where $M^3$ is a Thurston geometry and $\tilde{G}$ a co-compact lattice. Then, as in the previous case, $M^3$ is the universal orbifold-covering space of $\tilde{X}$. Moreover, the canonical projection $\tilde{X}\rightarrow X$ is an orbifold-covering map. Therefore, $M^3$ covers $X$  with deck transformation group $G$  and it follows that we can put on $X$ the quotient metric on $M^3/G$. Thus we conclude that $X$ is geometric.
\end{proof}

\noindent \textbf{Proof of Remark~\ref{R:HYP_GEOM_NOT}.} To conclude this section, we explain why a closed collapsing Alexandrov $3$-space cannot admit hyperbolic geometry. Let $X$ be a closed, collapsing Alexandrov $3$-space and suppose that $X$ admits hyperbolic geometry. We have two cases to consider: either $X$ is a $3$-manifold or $X$ has topological singularities. 

Suppose first that $X$ is a $3$-manifold. Then, by \cite[Corollary 1.7]{MitYam2} (after passing via the orientable double cover of $X$ if necessary), the simplicial volume $\|X\|$ of $X$ is zero. On the other hand, since $X$ admits a hyperbolic Riemannian metric, by \cite[Theorem 6.2]{Thurston1978} (after passing via the orientable double cover of $X$ if necessary), the simplicial volume of $X$ is non-zero, which is a contradiction. 

Suppose now that $X$ is not a $3$-manifold. Then its orientable double branched cover $\tilde{X}$ is an orientable topological $3$-manifold. By  Corollary~\ref{COR:COLLAPSE}, $\tilde{X}$ admits a sequence of collapsing Alexandrov metrics. Hence, by {\cite[Corollary 1.7]{MitYam2}, the simplicial volume of $\tilde{X}$ is zero. Now, since $X$ admits hyperbolic geometry, $\tilde{X}$ is also hyperbolic, so its simplicial volume cannot be zero, which is a contradiction. \hfill $\square$

%------------------------------------------------------
% 		SECTION: TWO-DIMENSIONAL LIMIT SPACE
%------------------------------------------------------

\section{Two-dimensional limit space}
\label{S:TWO_DIM_LIMIT}

In this section we deal with the case that $\{X_i\}_{i=1}^{\infty}$ collapses to $Y$ with $\dim Y=2$. In this case the homeomorphism type of $X_i$, for sufficiently big $i$, is described in Table \ref{TBL:2-DIM_LIMIT}. We begin by stating a technical lemma which we will need throughout this section. Let us denote the connected sum of $m\geq 0$ copies of a space $X$ by $\csum_m X$.

\begin{lem}
\label{L:IRREDUCIBLE_IMPLIES_PRIME}
Let $M$ be a closed $3$-manifold and let $X$ be an Alexandrov $3$-space homeomorphic to $M\#\left(\csum_r \Susp(\mathbb{R}P^2)\right)$, for $r\geq 0$. If every embedded $2$-sphere bounds a $3$-ball,  then one of the following holds:
\begin{itemize}
\item[(i)] $M$ is prime and $X$ is homeomorphic to $M$,
\item[(ii)] $X$ is homeomorphic to $\Susp(\mathbb{R}P^2)$.
\end{itemize}
\end{lem}

% PROOF

\begin{proof}
If $X$ is a topological manifold then $r=0$ and the result is the classical fact that an irreducible $3$-manifold is prime. Therefore we assume that the set of topologically singular points of $X$ is non-empty. Let us denote $Y=\csum_r \Susp(\mathbb{R}P^2)$. Now, let $S\subset M\# Y$ be the $2$-sphere which resulted from the connected sum operation. By hypothesis $S$ must bound a $3$-ball. Since $Y$ contains topologically singular points, this implies that $M$ is homeomorphic to $\mathbb{S}^3$. Hence, $X$ is homeomorphic to $Y$. Now, let us denote $Y_{r-1}= \csum_{r-1} \Susp(\mathbb{R}P^2)$ so that $X\cong \Susp(\mathbb{R}P^2)\# Y_{r-1}$. Let $\bar{S}\subset \Susp(\mathbb{R}P^2)\# Y_{r-1}$ be the $2$-sphere arising from the connected sum operation. Then, $\bar{S}$ must bound a $3$-ball. As in the previous case, the presence of topologically singular points implies that this is only possible if $r-1=0$, i.e., if $X\cong \Susp(\mathbb{R}P^2)$.
\end{proof}

We divide the following analysis in two cases depending whether $X_i$ has fibers of type $\bpt$ or not.

% SUBSECTION

\subsection*{$X_i$ does not contain fibers of type $\bpt$} 
 Under the assumption of no fibers of type $\bpt$, $X_i$ is homeomorphic to a usual Seifert fibered space (possibly with attached generalized solid tori and generalized Klein bottles). Then by  \cite[Corollary 6.2]{GalNun} $X_i$ is homeomorphic to an Alexandrov space admitting an isometric and effective, local circle action. Therefore, by \cite[Theorem B]{GalNun} (cf. Section \ref{S:PRELIM}), the aforementioned isometric local circle action is given by the set of invariants 
\[
\{b; \varepsilon,g,(f,k_1),(t,k_2),(s,k_3); \{(\alpha_i,\beta_i)\}_{i=1}^{n};(r_1, \ldots, r_{s-k_3});(q_1, q_2, \ldots, q_{k_3} )\}.
\]
In the following we exhaust every possible value of the invariants to conclude that, in each case, $X_i$ is geometric. For simplicity we will use the notation $k:=k_1+k_2+k_3$.  

\begin{case}[$k=0$, $f>0$]
 The topological decomposition of Theorem \ref{THM:INVARIANTS} applied to this case yields that, depending on the values of the invariants, $X_i$ is homeomorphic to either 
\[
\mathbb{S}^3\#
\left(\csum_{\varphi}\mathbb{S}^2\times \mathbb{S}^1 \right)\#
\left(\csum_t\mathbb{R}P^2\times \mathbb{S}^1 \right)\#
\left(\csum_{j=1}^n L(\alpha_j,\beta_j)\right)\#
\left( \csum_{r} \mathrm{Susp}(\mathbb{R}P^2) \right),
\]
or 
\[
\mathbb{S}^2\tilde{\times}\mathbb{S}^1\#
\left(\csum_{\varphi}\mathbb{S}^2\times \mathbb{S}^1 \right)\#
\left(\csum_{j=1}^n L(\alpha_j,\beta_j)\right)\#
\left( \csum_{r} \mathrm{Susp}(\mathbb{R}P^2) \right),
\]
where the value of $\varphi$ depends on $f$, $g$, $\varepsilon$ and $t$. The exact value of $\varphi$ does not play a role in our arguments (see \cite{GalNun}, Raymond, Fintushel for the precise definition).
Therefore,  since $X_i$ is irreducible, Lemma \ref{L:IRREDUCIBLE_IMPLIES_PRIME} implies that only one connected summand of the previous equivariant connected sum decompositions can appear. We observe that the possible connected summands are all geometric. More explicitly, $\mathbb{S}^3$, lens spaces and $\mathrm{Susp}(\mathbb{R}P^2)$ admit $\mathbb{S}^3$-geometry while $\mathbb{S}^2\times \mathbb{S}^1$, $\mathbb{S}^2\tilde{\times} \mathbb{S}^1$, $\mathbb{R}P^2\times \mathbb{S}^1$ admit the $(\mathbb{S}^2\times \mathbb{R})$-geometry. 
\end{case}

%%%%%%%%%%%%%%%%%%%%%%%%%%%%%%%%%%%%%%%%%%%%%%%%%%%%%%%%%%%%%%%%%%%%

\begin{case}[$k=0$, $f=0$, $t=0$]

In this case, the condition $f=0$ implies that the Alexandrov spaces $X_i$ considered are homeomorphic to topological manifolds since any action by isometries on a non-manifold Alexandrov space has fixed points (the $SF$-points). Moreover, the conditions $f=t=0$ imply that $X_i$ is a Seifert manifold (see for example \cite[Page 150, Lines 16--20]{OrlRay}). This in turn implies that $X_i$ is geometric by \cite[Theorem 5.3, (ii)]{Sco}. The possible geometries appearing in this case are $\mathbb{S}^3$, $\mathbb{R}^3$, $\mathbb{S}^2\times \mathbb{R}$, $\mathbb{H}^2\times \mathbb{R}$, $\SL$, or $\Nil$.
\end{case}

%%%%%%%%%%%%%%%%%%%%%%%%%%%%%%%%%%%%%%%%%%%%%%%%%%%%%%%%%%%%%%%%%%%%

\begin{case}[$k=0$, $f=0$, $t>0$] {As in the previous case, the $X_i$ are homeomorphic to topological manifolds since $f=0$. By the definition of the local circle action invariants, in this case we have that $b=0$. More explicitly, $X_i$ is determined by the set of invariants 
\[
\{0; (\varepsilon, g, 0, t); \{(\alpha_i,\beta_i)\}_{i=1}^n  \}.
\]

Furthermore, since $t>0$, the $X_i$ are non-orientable. Therefore, by the proof of \cite[Theorem 3]{OrlRay}, the orientable double cover $\tilde{X}_i$ of $X_i$ is determined by the set of invariants
\[
\{-n; (\tilde{\varepsilon}, \tilde{g}, 0, 0); (\alpha_1,\beta_1), \{(\alpha_i, \alpha_i-\beta_i)\}_{i=1}^n \}.
\] 
Here, $\tilde{\varepsilon}=o_1$ if $\varepsilon\in\{o_1,o_2,n_1\}$ and $\tilde{\varepsilon}=n_2$ if $\varepsilon\in\{n_2, n_3, n_4\}$. The value of $\tilde{g}$ is $2g+t-1$ if $\varepsilon\in\{o_1,o_2,n_2\}$, it is $g+t-1$ if $\varepsilon=n_1$ and it is $g+t-2$ if $\varepsilon\in\{n_3, n_4\}$.

We observe then that $\tilde{X}_i$ falls under the considerations of the previous case and therefore is a Seifert manifold. Therefore $X_i$ is doubly covered by a geometric manifold and hence it itself is geometric, the possible geometries being the same as in the last case. 
} 
\end{case}

%%%%%%%%%%%%%%%%%%%%%%%%%%%%%%%%%%%%%%%%%%%%%%%%%%%%%%%%%%%%%%%%%%%%

\begin{case}[$k\neq 0$, $f>0$]

In this case the classes $o_1$, $o_2$ merge into a single class $o$ while $n_1$, $n_2$, $n_3$ and $n_4$ merge into one class $n$. The decomposition of Theorem \ref{THM:INVARIANTS} in this case is the following. If $t>0$, then $X_i$ is homeomorphic to 
\[
\left(\csum_{\varphi+1}(\mathbb{S}^2\times \mathbb{S}^1)\right)\#\left( \csum_{t}\mathbb{R}P^2\times \mathbb{S}^1 \right)\#\left(\csum_{j=1}^n L(\alpha_j,\beta_j)\right)\# \left( \csum_{r} \mathrm{Susp}(\mathbb{R}P^2) \right).
\]
If $t=0$, then $X_i$ is homeomorphic to 
\[
\left(\mathbb{S}^2\tilde{\times}\mathbb{S}^1\right)\!\#\! \left(\!\csum_{\varphi}\mathbb{S}^2\times \mathbb{S}^1\!\right)\!\#\!\left(\! \csum_{t}\mathbb{R}P^2\times \mathbb{S}^1\! \right)\!\#\!\left(\csum_{j=1}^n L(\alpha_j,\beta_j)\right)\!\#\! \left( \csum_{r} \mathrm{Susp}(\mathbb{R}P^2)\! \right).
\]
Here, the value of $\varphi$ is determined by $f$, $g$ and $\varepsilon$, but the precise value is not necessary for our analysis.

As in the case $k=0, f>0$, since $X_i$ is irreducible, Lemma \ref{L:IRREDUCIBLE_IMPLIES_PRIME} implies that only one connected summand can appear, and in turn this yields that $X_i$ is geometric. 
\end{case}

%%%%%%%%%%%%%%%%%%%%%%%%%%%%%%%%%%%%%%%%%%%%%%%%%%%%%%%%%%%%%%%%%%%%

\begin{case}[$k\neq 0$, $f=0$]
As in the previous cases in which $f=0$, the $X_i$ are homeomorphic to topological manifolds. If $t=0$, then $X_i$ is a Seifert manifold, whereas if $t>0$ its orientable double cover is a Seifert manifold. In both cases the space is geometric.
\end{case}

%%%%%%%%%%%%%%%%%%%%%%%%%%%%%%%%%%%%%%%%%%%%%%%%%%%%%%%%%%%%%%%%%%%%

\subsection*{$X_i$ contains fibers of type $\bpt$}

In this case, for each $X_i$ we consider its branched orientable double cover $\tilde{X}_i$. We observe that $\tilde{X}_i$ does not contain fibers of type $\bpt$ since its subset of topologically singular points is empty. Therefore, by \cite[Corollary 6.2]{GalNun} $\tilde{X}_i$ is homeomorphic to a closed Alexandrov $3$-manifold admitting an isometric local circle action determined by a set of invariants 
\[
\{\tilde{b}; \tilde{\varepsilon},\tilde{g},(\tilde{f},\tilde{k}_1),(\tilde{t},\tilde{k}_2); \{(\tilde{\alpha}_i,\tilde{\beta}_i)\}_{i=1}^{n}\}.
\]
However, since $\tilde{X}_i$ is orientable, we have that $\tilde{t}=0$. The cases having $\tilde{f}=0$ are settled as in the case in which $X_i$ does not contain fibers of type $\bpt$. Hence, in the following we assume that $\tilde{f}>0$. Under this assumption, Theorem \ref{THM:INVARIANTS} yields that $\tilde{X}_i$ is homeomorphic to 
\begin{equation}
\label{EQ:EQUIVARIANT_DECOMPOSITION_DOUBLE_COVER}
\mathbb{S}^3\#\left(\csum_{\varphi}\mathbb{S}^2\times \mathbb{S}^1\right)\# \left(\csum_{j=1}^n L(\alpha_j,\beta_j)\right).
\end{equation}

We point out that $\mathbb{S}^2\tilde{\times}\mathbb{S}^1$ does not appear in the decomposition as a consequence of the orientability of $\tilde{X}_i$. Let us denote by $\iota_i:\tilde{X}_i\to \tilde{X}_i$ the involution such that $\tilde{X}_i/\iota_i$ is homeomorphic to $X_i$. Then, $\iota_i$ induces an involution on the connected sum \ref{EQ:EQUIVARIANT_DECOMPOSITION_DOUBLE_COVER} and we can equip the space with a Riemannian metric which is invariant with respect to $\iota_i$.

 We now observe that, by \cite{DinkLeeb}, there is an equivariant (with respect to $\iota_i$) Ricci flow with surgery on $\tilde{X}_i$. Furthermore, by \cite[Theorem 1.1]{Per3} the Ricci flow goes extinct in finite time. Now, \cite[Corollary 4.5]{Dink} implies that the action of $\iota_i$ is an equivariant connected sum of standard actions on components diffeomorphic to spherical space forms, $\mathbb{S}^2\times\mathbb{S}^1$ or $\mathbb{R}P^3\#\mathbb{R}P^3$. In particular, the connected summands $ L(\alpha_1,\beta_1)\#\cdots \# L(\alpha_n,\beta_n)$ and $\csum_{\varphi}\mathbb{S}^2\times \mathbb{S}^1$ in \eqref{EQ:EQUIVARIANT_DECOMPOSITION_DOUBLE_COVER} are invariant under the action of $\iota_i$.

 We now observe the following. If the connected sum decomposition \eqref{EQ:EQUIVARIANT_DECOMPOSITION_DOUBLE_COVER} of $\tilde{X}_i$ contains both lens space $L(\alpha_j,\beta_j)$ and $(\mathbb{S}^2\times \mathbb{S}^1)$-summands, then the $2$-sphere $S$ dividing $\tilde{X}_i$ into the summands  $\csum_{j=1}^n L(\alpha_j,\beta_j)$ and $\csum_{\varphi}\mathbb{S}^2\times \mathbb{S}^1$ can be taken to  be invariant under the the action of $\iota_i$. 
By the classification of the involutions of the $2$-sphere, $S/\iota_i\subset X_i$ is either homeomorphic to $\mathbb{R}P^2$ or to a $2$-sphere. If $S/\iota_i \cong \mathbb{R}P^2$, then we observe that $S/\iota_i$ does not bound a cone over $\mathbb{R}P^2$. Similarly, if $S/\iota_i$ is a $2$-sphere then it does not bound a $3$-ball, contradicting the irreducibility of $X_i$. Therefore, we can assume that 
$\tilde{X}_i$ is either a connected sum of lens spaces or a connected sum of copies of $\mathbb{S}^2\times \mathbb{S}^1$.

\begin{case}[$\tilde{X_i}$ is homeomorphic to $L(\alpha_1,\beta_1)\#\ldots \# L(\alpha_n,\beta_n)$ ]
By \cite[Corollary 3]{KimToll} $\iota_i= h_{i_1}\#\ldots\#h_{i_m}$ and $\tilde{X}_i=M_{i_1}\#\ldots\#M_{i_m}$, where each $h_{i_j}$ is an involution on $M_{i_j}=A_{i_j}\#Q_{i_j}\#A_{i_j}^{\epsilon}$ arising from operation I-1 (see \cite[Page 260]{KimToll}. Here, $A_{i_j}$ is a connected sum of lens spaces, $A_{i_j}^{\epsilon}$ is either $A_{i_j}$ or $A_{i_j}$ with the orientation reversed, and $Q_{i_j}$ is either $\mathbb{R}P^3$ or $\mathbb{S}^3$. 
Moreover, since the involution $\iota_i\colon \tilde{X}_i\to\tilde{X}_i$ has fixed points, at least one of the $h_{i_j}$ must have fixed points. This is the involution we will consider in what follows and fix the index $i_j$.

If $Q_{i_j}$ is $\mathbb{S}^3$, then the involution $h_{i_j}$  interchanges the summands $A_{i_j}$ and $A_{i_j}^{\epsilon}$. In particular $Q_{i_j}=\mathbb{S}^3$ is invariant under $h_{i_j}$ and, since the involution has isolated fixed points, its restriction to $Q_{i_j}=\mathbb{S}^3$ corresponds to the suspension of the antipodal map on $\mathbb{S}^2$. The quotient of $A_{i_j}\#Q_{i_j}\#A_{i_j}^{\epsilon}$ by $h_{i_j}$ is homeomorphic to $A_{i_j}\#\Susp(\mathbb{R}P^2)$, where the $2$-sphere $S$ dividing the connected sum is the projection of the two $2$-spheres $S_1,S_2$ that divide the connected sum $A_{i_j}\#Q_{i_j}\#A_{i_j}^{\epsilon}$ into three summands. Therefore, the sphere $S$ does not bound a $3$-ball, contradicting the irreducibility of $X_i$. 

 If $Q_i$ is $\mathbb{R}P^3$ then, we recall that by \cite{Kwun} there is exactly one orientation reversing involution on $\mathbb{R}P^3$. However, this involution has a $2$-dimensional fixed point set. Since $\iota_i$ has only isolated fixed points this is a contradiction.
\end{case}

\begin{case}[$\tilde{X}_i$ is homeomorphic to $\csum_{\varphi}\mathbb{S}^2\times \mathbb{S}^1$]
Let us recall that by \cite{KimToll} every involution on a connected sum of closed $3$-manifolds is constructed by successive application of the so called four \textit{I-operations} (see \cite[Page 260]{KimToll}). In particular $\iota_i$ can be described in such fashion. We will split our analysis in four cases depending on the type of the last I-operation used to construct $\iota_i$. In each case we will conclude that if $X_i$ is homeomorphic to $\csum_{\varphi}\mathbb{S}^2\times \mathbb{S}^1$, then $\varphi=1$. In the remainder of this case, we follow the notation of \cite[Page 260]{KimToll}. 

%%%%%%%%%%%%%%%%%%%%%%%%%%%%%%%%%%%%%%%%%%%%%%%%%%%%%%%%%%%%%%%%%%%%

\begin{subcase}[I-1 operation]
 Here, $M_1$ and $M_2$ are connected sums of copies of $\mathbb{S}^2\times \mathbb{S}^1$. Let $C_1\subset M_1$ be the cell used to perform the I-1 operation. We let $\pi:\tilde{X}_i\to X_i$ be the canonical projection. Then $\pi(\partial C_1)$ is a $2$-sphere in $X_i$. Then, $\pi(\partial C_1)$ bounds a $3$-ball in $X_i$ if and only if either $M_1/\iota_i$ or $M_2/\iota_i$ are homeomorphic to a closed $3$-ball. Therefore, the question reduces to see if the quotient of a connected sum of copies of $\mathbb{S}^2\times \mathbb{S}^1$ can be homeomorphic to a $3$-ball. To address this question, we observe that if the set of fixed points of $\iota_i$ is empty, then $\partial( M_1/\iota_i) =\emptyset$. Then, $M_1/\iota_i$ cannot be homeomorphic to a $3$-ball. Therefore we assume that $\iota_i$ has fixed points on $M_1$. Since the fixed points of $\iota_i$ correspond to topologically singular points of $X_i$ under $\pi$, the quotient $M_1/\iota_i$ cannot be homeomorphic to a $3$-ball. Hence, we conclude that  $\pi(\partial C_1)$ does not bound a $3$-ball in $X_i$. This is a contradiction since $X_i$ is irreducible.
\end{subcase}

%%%%%%%%%%%%%%%%%%%%%%%%%%%%%%%%%%%%%%%%%%%%%%%%%%%%%%%%%%%%%%%%%%%%

\begin{subcase}[I-2 operation]
Let $S$ be the sphere used for the connected sum. Then $\pi(S)$ is homeomorphic to an $\mathbb{R}P^2$ in $X_i$ which does not bound a cone over $\mathbb{R}P^2$. Otherwise, if we assume without loss of generality that $M_1/h_1$ is homeomorphic to $K(\mathbb{R}P^2)$ then $M_1$ would be a $3$-ball, implying that the connected sum $M_1\# M_2$ is trivial. This is a contradiction. 
\end{subcase} 

%%%%%%%%%%%%%%%%%%%%%%%%%%%%%%%%%%%%%%%%%%%%%%%%%%%%%%%%%%%%%%%%%%%%

\begin{subcase}[I-3 operation]
In this case, the quotient of the involution has the form $(M_1/h_1) \# (M_2/h_2)$. If both $h_1$ and $h_2$ have fixed points, then it is clear that $M_i/h_i$ is not homeomorphic to a $3$-ball. Therefore, we now assume, without loss of generality, that the fixed point set of $h_1$ is empty. Then, if $M_1/h_1$ is homeomorphic to a $3$-ball, it follows that $M_1$ is homeomorphic to a disjoint union of two $3$-balls. In particular $M_1$ is disconnected, which is a contradiction.
\end{subcase}

%%%%%%%%%%%%%%%%%%%%%%%%%%%%%%%%%%%%%%%%%%%%%%%%%%%%%%%%%%%%%%%%%%%%

\begin{subcase}[I-4 operation]
The quotient of the involution in this case is constructed as follows. First, we take out small neighborhoods of two topologically singular points in $X_i$. Then, we glue $\mathbb{R}P^2\times I$ to the resulting space along the $\mathbb{R}P^2$ boundaries. Therefore, the space obtained is not irreducible since it contains a $2$-sided $\mathbb{R}P^2$ which does not bound a $K(\mathbb{R}P^2)$. This contradicts the irreducibility of $X_i$.  
\end{subcase}

%%%%%%%%%%%%%%%%%%%%%%%%%%%%%%%%%%%%%%%%%%%%%%%%%%%%%%%%%%%%%%%%%%%%
\end{case}

%------------------------------------------------------
% 		SECTION: ONE-DIMENSIONAL LIMIT SPACE
%------------------------------------------------------

\section{One-dimensional limit space}
\label{S:ONE_DIM_LIMIT}

 In this section we analyze the case in which the sequence $\{X_i\}_{i=1}^{\infty}$ collapses to a  one-dimensional space $Y$. We distinguish two cases, depending on whether $\partial Y=\emptyset$ or $\partial Y\neq \emptyset$. The homeomorphism type of $X_i$ is described in Table \ref{TBL:1-DIM_LIMIT}. We also point out that, in this section, the assumption of irreducibility is not needed to assert that $X_i$ is geometric. Under such an assumption, at least Cases \ref{SUBCASE:RP3-intD3UBS2}, \ref{SUBCASE:RP3-intD3URP3-intD3}, and  \ref{SUBCASE:BS2UBS2} are ruled out. However, in those cases the same techniques we use throughout the section yield that $X_i$ is geometric and we include them as well. We also observe that in most cases we do not indicate all the possible geometries that the space in question admits. 

\subsection*{$Y$ does not have boundary}

Under this assumption, $X_i$ is homeomorphic to an $F$-fiber bundle $F\hookrightarrow E\stackrel{_\rho}{\to} \mathbb{S}^1$. The possible fibers $F$ are $K^2$, $T^2$, $\mathbb{S}^2$ or $\mathbb{R}P^2$. The universal cover of $E$ is calculated in the following way. We let $\xi:\mathbb{R}\to \mathbb{S}^1$ be the universal cover projection. Now, we take the pullback $\tilde{\rho}:\xi^{*}E\to \mathbb{R}$ of $\rho: E\to \mathbb{S}^1$ with respect to $\xi$ and observe that $\xi^{*}E$ is a trivial bundle since the base $\mathbb{R}$ is contractible. Furthermore, the fiber of $\xi^{*}E$ is homeomorphic to $F$ and therefore $\xi^{*}E \cong F\times \mathbb{R}$.  Moreover, the induced map $\tilde{\xi}:\xi^{*}E\to E$ is a covering map, and we conclude that $E$ is covered by $F\times \mathbb{R}$. 

If $F= K^2$ or $T^2$ then the universal cover of $\xi^{*}E$ (and consequently of $E$) is $\mathbb{R}^2\times \mathbb{R}\cong \mathbb{R}^3$. We further observe that using the usual double covers $T^2\to K^2$ and $\mathbb{S}^2\to \mathbb{R}P^2$ it follows that $T^2\times\mathbb{R}$ covers any torus and Klein-bottle bundles over $\mathbb{S}^1$ and that $\mathbb{S}^2\times\mathbb{R} $ covers any $\mathbb{S}^2$ and $\mathbb{R}P^2$ bundles over $\mathbb{S}^1$. On the other hand, if $F=\mathbb{S}^2$ or $\mathbb{R}P^2$, then the universal cover is $\mathbb{S}^2\times \mathbb{R}$. Hence, $E$ is geometric since its universal cover is geometric for any $F$. We now point out the precise geometries that $E$ admits depending on $F$. 

%%%%%%%%%%%%%%%%%%%%%%%%%%%%%%%%%%%%%%%%%%%%%%%%%%%%%%%%%%%%%%%%%%%%

\begin{case}[$F=\mathbb{S}^2$]
There are two such bundles: $\mathbb{S}^2\times \mathbb{S}^1$ and $\mathbb{S}^2\tilde{\times}\mathbb{S}^1$. Both of them are covered by $\mathbb{S}^2\times \mathbb{S}^1$. Therefore both admit $\mathbb{S}^2\times \mathbb{R}$-geometry.
\end{case}

%%%%%%%%%%%%%%%%%%%%%%%%%%%%%%%%%%%%%%%%%%%%%%%%%%%%%%%%%%%%%%%%%%%%

\begin{case}[$F=T^2$]
We describe $E$ as a mapping torus over $F$ in the following way. Let $S^1\setminus \{\mathrm{pt}\}$ be identified with $I=[a,b]$. Then, the bundle $\rho|_{\rho^{-1}(I)}\colon\rho^{-1}(I)\to I$ is trivial since $I$ is contractible. Therefore, $\rho^{-1}(I)\cong F\times I$. Hence, the space $E$ is obtained as $F\times I/h$ where $h\colon\rho^{-1}(a)\to\rho^{-1}(b)$ is a homeomorphism.  

 Recall that homotopic gluing maps $h$ give rise to isomorphic mapping tori. Then,  \cite[Theorem 13.2, Theorem 13.5]{FarbMarg} yield that $E$ may admit the $\mathbb{R}^3$-, $\Nil$- or $\Sol$-geometries depending on the gluing map $h$.
\end{case}

%%%%%%%%%%%%%%%%%%%%%%%%%%%%%%%%%%%%%%%%%%%%%%%%%%%%%%%%%%%%%%%%%%%%

\subsection*{$Y$ has boundary}  In this case $X_i$ is homeomorphic to a space obtained by gluing two pieces $B$ and $B'$ having as possible boundaries $\mathbb{S}^2$, $\mathbb{R}P^2$, $T^2$ or $K^2$. We analyze each possible gluing $B\cup B'$. Since sometimes we will need to compute the double branched cover of such gluing, we will use the following obvious lemma.

\begin{lem} 
\label{L:DBL_COVER_GLUING}
Let $X$ be a non-manifold Alexandrov $3$-space obtained by gluing two pieces $B$ and $B'$ having as possible boundaries $\mathbb{S}^2$, $\mathbb{R}P^2$, $T^2$ or $K^2$ by a map $\varphi$. Let $\pi\colon\tilde{X}\to X$ be the orientable double branched cover of $X$. Then the following hold: 
\begin{itemize}
	\item[\emph{(i)}] If only one of the pieces, say, $B$, has topological singularities, then $\pi^{-1}(B)$ has an isometric involution with only isolated fixed points, $\partial(\pi^{-1}(B))$ is a $2$-fold cover of $\partial B$, and $\pi^{-1}(B')$ is a two-fold cover of $B'$. 
	\item[\emph{(ii)}] If both pieces $B$ and $B'$ have topological singularities, then $\pi^{-1}(B)$ and $\pi^{-1}(B')$ have isometric involutions with only isolated fixed points and $\partial(\pi^{-1}(B))$ is a $2$-fold cover of $\partial B$. 
	\item[\emph{(iii)}] If a piece $B$ is an orientable manifold, then $\pi^{-1}(B)$ is given by two copies of $B$, and the involution sends one copy to the other;  each copy of $B$ is glued to $\pi^{-1}(B')$ by a copy of $\varphi$.
	\item[\emph{(iv)}] If a piece $B$ is a non-orientable manifold, then $\pi^{-1}(B)$ is given by the orientable double cover of $B$, and it is glued to $\pi^{-1}(B')$ by a two-fold cover of $\varphi$.
\end{itemize}
\end{lem}

%%%%%%%%%%%%%%%%%%%%%%%%%%%%%%%%%%%%%%%%%%%%%%%%%%%%%%%%%%%%%%%%%%%%

\begin{case}[$\partial B=\mathbb{S}^2$]
The possible pieces satisfying this assumption are $D^3$, $\mathbb{R}P^3\setminus \mathrm{int}D^3$ and $\bstwo$. We now examine each possible combination of these pieces. 

%%%%%%%%%%%%%%%%%%%%%%%%%%%%%%%%%%%%%%%%%%%%%%%%%%%%%%%%%%%%%%%%%%%%

\begin{subcase}[$D^3\cup D^3$] In this case, $X$ is homeomorphic to $\mathbb{S}^3$, which admits spherical geometry.
\end{subcase}

%%%%%%%%%%%%%%%%%%%%%%%%%%%%%%%%%%%%%%%%%%%%%%%%%%%%%%%%%%%%%%%%%%%%

\begin{subcase}[$D^3\cup \left(\mathbb{R}P^3 \setminus \mathrm{int} (D^3)\right)$]
We have that $\mathbb{R}P^3\setminus \mathrm{int}(D^3) \cup_{\partial} D^3\cong \mathbb{R}P^3$, which admits spherical geometry.
\end{subcase}

%%%%%%%%%%%%%%%%%%%%%%%%%%%%%%%%%%%%%%%%%%%%%%%%%%%%%%%%%%%%%%%%%%%%

\begin{subcase}[$D^3\cup \bstwo$]
By \cite[Remark 2.62]{MitYam}, $\bstwo$ is homeomorphic to $K_1(\mathbb{R}P^2)\cup_{\mathrm{Mo}}K_1(\mathbb{R}P^2)$. Using a similar argument to \cite[Lemma 2.61]{MitYam}, it can be proved that the topology of $\bstwo$ does not depend on the gluing map $\mathrm{Mo}\to \mathrm{Mo}$ used. Therefore, using the identity map $\mathrm{id}\colon\mathrm{Mo}\to \mathrm{Mo}$ we obtain that $\bstwo \cong \mathrm{Susp}(\mathbb{R}P^2)\setminus D^3$. Hence, $D^3\cup \bstwo \cong \mathrm{Susp}(\mathbb{R}P^2)$ which admits spherical geometry.  
Alternatively, we can use Lemma~\ref{L:DBL_COVER_GLUING}. The description of $\bstwo$ implies that $\pi^{-1}(\bstwo)$ is homeomorphic to $\mathbb{S}^2\times [-1,1]$ and the involution sends the point $(x,t)$ to $(\sigma(x),-t)$, where $\sigma:\mathbb{S}^2\to\mathbb{S} ^2$ is the suspension of the antipodal map on $\mathbb{S}^1$. On the other hand, $\pi^{-1}(B^3)$ consists of two copies of $B^3$. Hence $\tilde{X}$ is $\mathbb{S}^3$ where the involution sends $(x,y,z,t)$ to $(x,-y,-z,-t)$. This is the suspension of the antipodal map in the hyperplane $(y,z,t)$. Hence the quotient is $\Susp(\mathbb{R}P^2)$.
\end{subcase}

%%%%%%%%%%%%%%%%%%%%%%%%%%%%%%%%%%%%%%%%%%%%%%%%%%%%%%%%%%%%%%%%%%%%

\begin{subcase}[$\left(\mathbb{R}P^3\setminus \mathrm{int}(D^3)\right)\cup \bstwo$]
\label{SUBCASE:RP3-intD3UBS2}
By \cite[Remark 2.62]{MitYam}, this space is homeomorphic to $\mathbb{R}P^3\# \mathrm{Susp}(\mathbb{R}P^2)$. Since this space, as the ones in the next two cases, is not irreducible, it does not fall into the family of spaces we are considering. For completeness, however, we will show that it is geometric.
We now obtain the orientable branched double cover. We take out small cones centered at each of the two topologically singular points of $\mathbb{R}P^3\# \mathrm{Susp}(\mathbb{R}P^2)$. The resulting space is homeomorphic to $\mathbb{R}P^3\#(\mathbb{R}P^2\times [-1,1])$. Now, since $\mathbb{R}P^3$ is orientable, the orientable double cover of $\mathbb{R}P^3 \setminus D^3$ is $\left(\mathbb{R}P^3\setminus D^3\right)\times \{-1,1\}$. On the other hand, the orientable double cover of $(\mathbb{R}P^2\times [-1,1])\setminus D^3$ is $(\mathbb{S}^2\times [-1,1])\setminus D_1^3\cup D^3_2$, where $D^3_1$ and $D^3_2$ are two disjoint $3$-balls.  Therefore, the orientable double cover of $\mathbb{R}P^3\#(\mathbb{R}P^2\times [-1,1])$ is 
\[
\left((\mathbb{R}P^3\setminus D^3)\times \{-1\}\right)\cup_{\partial D^3_1}\left((\mathbb{S}^2\times [-1,1])\setminus D_1^3\cup D^3_2\right)\cup_{\partial D^3_2}\left((\mathbb{R}P^3\setminus D^3)\times \{1\}\right),
\]
which by definition is $\mathbb{R}P^3\#\left( \mathbb{S}^2\times [-1,1]\right)\#\mathbb{R}P^3$. Gluing two $3$-balls along the two boundary components of this space, we obtain the branched orientable double cover $\mathbb{R}P^3\#(\mathbb{S}^3)\#\mathbb{R}P^3\cong \mathbb{R}P^3\#\mathbb{R}P^3$. The involution can be made isometric with respect to a non-negatively curved metric on $\mathbb{R}P^3\#\mathbb{R}P^3$ with $(\mathbb{S}^2\times\mathbb{R})$-geometry (see \cite[Section 5.2]{Dink}). Therefore, $\mathbb{R}P^3\# \mathrm{Susp}(\mathbb{R}P^2)$ is a closed Alexandrov space  admitting a metric with non-negative curvature. We point out that this space should be added to the list of closed, non-negatively curved Alexandrov $3$-spaces in Theorem~1.3 in the published version of \cite{GalGui}.

\end{subcase}

%%%%%%%%%%%%%%%%%%%%%%%%%%%%%%%%%%%%%%%%%%%%%%%%%%%%%%%%%%%%%%%%%%%%

\begin{subcase}[$\left(\mathbb{R}P^3\setminus \mathrm{int}(D^3)\right)\cup\left( \mathbb{R}P^3\setminus \mathrm{int}(D^3)\right)$]
\label{SUBCASE:RP3-intD3URP3-intD3}
 This case was considered in Case (1) of the proof of \cite[Theorem 0.6]{ShiYam}. By definition of the connected sum, $\mathbb{R}P^3\setminus \mathrm{int}(D^3)\cup_{\partial}\mathbb{R}P^3\setminus \mathrm{int}(D^3)=\mathbb{R}P^3\# \mathbb{R}P^3$. Therefore, this space admits $(\mathbb{S}^2\times \mathbb{R})$-geometry. Observe that this space is not irreducible.
\end{subcase} 
 
%%%%%%%%%%%%%%%%%%%%%%%%%%%%%%%%%%%%%%%%%%%%%%%%%%%%%%%%%%%%%%%%%%%%

\begin{subcase}[$\bstwo\cup \bstwo$]
\label{SUBCASE:BS2UBS2}
 It follows from \cite[Remark 2.62]{MitYam}, that this space is homeomorphic to $\mathrm{Susp}(\mathbb{R}P^2)\# \mathrm{Susp}(\mathbb{R}P^2)$. Therefore, this space admits the $(\mathbb{S}^2\times \mathbb{R})$-geometry (see the proof of \cite[Theorem 1.3]{GalGui}, Case 2), part (a) of the case on which $X$ is not a topological manifold). As in the preceding two cases, this space is not irreducible.
\end{subcase}

%%%%%%%%%%%%%%%%%%%%%%%%%%%%%%%%%%%%%%%%%%%%%%%%%%%%%%%%%%%%%%%%%%%%
\end{case}

%%%%%%%%%%%%%%%%%%%%%%%%%%%%%%%%%%%%%%%%%%%%%%%%%%%%%%%%%%%%%%%%%%%%

\begin{case}[$\partial B=\mathbb{R}P^2$]
The only piece satisfying this requirement is $K_1(\mathbb{R}P^2)$. Therefore, by \cite[Lemma 2.61]{MitYam}, the only space obtained here is $\mathrm{Susp}(\mathbb{R}P^2)$ which has spherical geometry. \\

\end{case}

%%%%%%%%%%%%%%%%%%%%%%%%%%%%%%%%%%%%%%%%%%%%%%%%%%%%%%%%%%%%%%%%%%%%

\begin{case}[$\partial B=T^2$]
The collection of pieces that meet this condition is $D^2\times \mathbb{S}^1$, $\mathrm{Mo}\times \mathbb{S}^1$, $K^2\widetilde{\times}I$, $\bsfour$.

%%%%%%%%%%%%%%%%%%%%%%%%%%%%%%%%%%%%%%%%%%%%%%%%%%%%%%%%%%%%%%%%%%%%

\begin{subcase}[$(D^2\times \mathbb{S}^1) \cup (D^2\times \mathbb
{S}^1)$]
 This case was considered in Case (2-i) of the proof of \cite[Theorem 0.6]{ShiYam}. Gluing two copies of a solid torus $D^2\times \mathbb{S}^1$ by a homeomorphism of their boundaries we obtain, by definition,  all possible lens spaces. Therefore this space admits either $\mathbb{S}^3$- or $(\mathbb{S}^2\times \mathbb{R})$-geometry (in the case of $\mathbb{S}^2\times\mathbb{S}^1$).
\end{subcase}

%%%%%%%%%%%%%%%%%%%%%%%%%%%%%%%%%%%%%%%%%%%%%%%%%%%%%%%%%%%%%%%%%%%%

\begin{subcase}[$(D^2\times \mathbb{S}^1) \cup_{\varphi} (\mathrm{Mo}\times \mathbb{S}^1)$]
\label{SUBCASE:D2xS1UMoxS1}
Let us consider the double cover
\[
\rho\colon \mathbb{S}^1\times I \times \mathbb{S}^1
% \cong T^2\times I 
\to \mathrm{Mo}\times \mathbb{S}^1 
\]
given by the involution $\iota\colon\mathbb{S}^1\times I \times \mathbb{S}^1\to \mathbb{S}^1\times I \times \mathbb{S}^1$ defined by $(x,t,y)\mapsto (-x,-t,y)$. 
Note that $\rho$ is the orientable double cover of $\mathrm{Mo}\times \mathbb{S}^1$ since it is a double cover and $\iota$ is orientation-reversing.  We now observe that since $D^2\times \mathbb{S}^1$ is orientable, its orientable double cover consists of two disjoint copies
of $D^2\times \mathbb{S}^1$.

By Lemma \ref{L:DBL_COVER_GLUING}, the orientable double cover of $(D^2\times \mathbb{S}^1) \cup_{\varphi} (\mathrm{Mo}\times \mathbb{S}^1)$ is obtained by gluing each copy of $D^2\times \mathbb{S}^1$ to a connected component of the boundary of $\mathbb{S}^1\times I \times \mathbb{S}^1$ via $\varphi$, thus resulting on $\mathbb{S}^2\times \mathbb{S}^1$; therefore this space admits $(\mathbb{S}^2\times \mathbb{R})$-geometry. 
\end{subcase}

%%%%%%%%%%%%%%%%%%%%%%%%%%%%%%%%%%%%%%%%%%%%%%%%%%%%%%%%%%%%%%%%%%%%

\begin{subcase}[$(D^2\times \mathbb{S}^1) \cup (K^2\widetilde{\times} I)$]
\label{Subcase:D2timesS1UK2tildeI}
This space was considered in Case (2-ii) of the proof of \cite[Theorem 0.6]{ShiYam}. The space $K^2\widetilde{\times}I$ is homeomorphic to $\mathrm{Mo}\tilde{\times}\mathbb{S}^1$. This implies that  $(D^2\times \mathbb{S}^1) \cup (K^2\widetilde{\times} \mathbb{S}^1)$ is homeomorphic to $D^2\times \mathbb{S}^1 \cup \mathrm{Mo}\widetilde{\times} \mathbb{S}^1$ which, depending on the gluing homeomorphism, is either a prism manifold, $\mathbb{S}^1\times \mathbb{S}^2$, or $\mathbb{R}P^3\# \mathbb{R}P^3$. Hence, the possible geometries for this case are the $\mathbb{S}^3$- and $(\mathbb{S}^2\times\mathbb{R})$-geometries. 
\end{subcase}

%%%%%%%%%%%%%%%%%%%%%%%%%%%%%%%%%%%%%%%%%%%%%%%%%%%%%%%%%%%%%%%%%%%%

\begin{subcase}[$(D^2\times \mathbb{S}^1) \cup \bsfour$]
\label{SUBCASE:D2timesS1UBS4}
By the description of $\bsfour$ in \cite[Corollary 2.56]{MitYam}, as a quotient of $T^2\times [-1,1]$ with respect to an orientation-reversing involution with only fixed points, the branched orientable double cover of $\bsfour$ is $T^2\times[-1,1]$. On the other hand, the orientable double cover of $D^2\times \mathbb{S}^1$ is a disjoint union $D^2\times \mathbb{S}^1 \sqcup D^2\times \mathbb{S}^1$.  By Lemma \ref{L:DBL_COVER_GLUING}, in the branched cover the gluing homeomorphisms $\partial (D^2\times \mathbb{S}^1)\to T^2\times \{1\}$ and $\partial (D^2\times \mathbb{S}^1)\to T^2\times \{-1\}$ 
coincide with each other. Therefore, 
\[
(D^2\times \mathbb{S}^1)\cup (T^2\times [-1,1])\cup (D^2\times \mathbb{S}^1) \cong (D^2\times \mathbb{S}^1) \cup_{\mathrm{id}} (D^2\times \mathbb{S}^1)\cong \mathbb{S}^2\times \mathbb{S}^1.
\]
 Thus $(D^2\times \mathbb{S}^1) \cup \bsfour$ admits $(\mathbb{S}^2\times \mathbb{R})$-geometry.
\end{subcase}

%%%%%%%%%%%%%%%%%%%%%%%%%%%%%%%%%%%%%%%%%%%%%%%%%%%%%%%%%%%%%%%%%%%%

\begin{subcase}[$(\mathrm{Mo}\times \mathbb{S}^1)\cup (\mathrm{Mo}\times \mathbb{S}^1)$]
\label{SUBCASE:MoxS1UMoxS1}
We consider $\rho: T^2\times I \to \mathrm{Mo}\times \mathbb{S}^1$, the double cover of Case \ref{SUBCASE:D2xS1UMoxS1}.  Once again, by Lemma \ref{L:DBL_COVER_GLUING},  the orientable cover of our space is then obtained by gluing pairwise the boundary components of two copies of $T^2\times I$;  this results in a fiber bundle over $\mathbb{S}^1$ with fiber a torus, and will admit an $\mathbb{R}^3$-, $\Nil$- or $\Sol$- geometry depending on the map $\varphi$.
\end{subcase}

%%%%%%%%%%%%%%%%%%%%%%%%%%%%%%%%%%%%%%%%%%%%%%%%%%%%%%%%%%%%%%%%%%%%

\begin{subcase}[$(\mathrm{Mo}\times \mathbb{S}^1) \cup (K^2\widetilde{\times} I)$]
\label{SUBCASE:MoxS1UK2tildeI}
Let us note that the canonical projection $\xi:\mathbb{S}^1\times I \times \mathbb{S}^1\to \mathrm{Mo}{\times}\mathbb{S}^1$ is the orientable double cover of this piece. 
On the other hand, $K^2\widetilde{\times} I$ being orientable, its orientable double cover is given by two  copies of itself. Thus the orientable double cover of $(\mathrm{Mo}\times \mathbb{S}^1) \cup (K^2\widetilde{\times} I)$ is obtained by gluing together both copies of 
$K^2\widetilde{\times} I$ along their boundary where the gluing map respects the $I$-bundle structure of each $K^2\widetilde{\times} I$ . 

Such space can be obtained as the quotient of
$\mathbb{T}^2\times \mathbb{S}^1$ by an isometric involution of the flat metric:
choose some involution $\sigma:\mathbb{T}^2\to\mathbb{T}^2$ with $\mathbb{T}^2/\sigma\simeq K^2$, and define the map
\[
\tau:\mathbb{T}^2\times \mathbb{S}^1\to\mathbb{T}^2\times \mathbb{S}^1, \quad \tau(p,z)=(\sigma(p), \bar{z}).
\] 
Therefore, $(\mathrm{Mo}\times \mathbb{S}^1) \cup (K^2\widetilde{\times} I)$ admits $\mathbb{R}^3$-geometry.
\end{subcase}

%%%%%%%%%%%%%%%%%%%%%%%%%%%%%%%%%%%%%%%%%%%%%%%%%%%%%%%%%%%%%%%%%%%%

\begin{subcase}[$(\mathrm{Mo}\times \mathbb{S}^1) \cup \bsfour$]
\label{SUBCASE:MoxS1UBS4}
{
As in case \ref{SUBCASE:D2timesS1UBS4}, 
 $\bsfour$ is isometric to a quotient of $T^2\times [-1,1]$ by an orientation reversing involution with only fixed points (cf. \cite[Corollary 2.56]{MitYam}). Therefore, in similar fashion to the previous two Cases, the space $(\mathrm{Mo}\times \mathbb{S}^1) \cup\bsfour$ is doubly covered by a $T^2$-bundle over $\mathbb{S}^1$ and therefore accepts $\mathbb{R}^3$-, $\Nil$ or $\Sol$-geometry depending on the gluing.
 }
\end{subcase}

%%%%%%%%%%%%%%%%%%%%%%%%%%%%%%%%%%%%%%%%%%%%%%%%%%%%%%%%%%%%%%%%%%%%

\begin{subcase}[$(K^2\widetilde{\times} I) \cup (K^2\widetilde{\times} I)$]
\label{SUBCASE:K2tildeIUK2tildeI}
This case was considered in Case (3) of the proof of \cite[Theorem 0.6, page 31]{ShiYam}, where it was observed that the space in question is doubly covered by a $T^2$-bundle over $\mathbb{S}^1$. Therefore, the space is geometric, with possible geometries $\mathbb{R}^3$, $\Nil$ or $\Sol$. The precise analysis is analogous to that of Case \ref{SUBCASE:MoxS1UMoxS1}.
\end{subcase}

%%%%%%%%%%%%%%%%%%%%%%%%%%%%%%%%%%%%%%%%%%%%%%%%%%%%%%%%%%%%%%%%%%%%

\begin{subcase}[$(K^2\widetilde{\times} I) \cup\bsfour$]
The analysis of this case is analogous to that of Case \ref{SUBCASE:MoxS1UK2tildeI}.
\end{subcase}

%%%%%%%%%%%%%%%%%%%%%%%%%%%%%%%%%%%%%%%%%%%%%%%%%%%%%%%%%%%%%%%%%%%%

\begin{subcase}[$\bsfour\cup\bsfour$]
Since each $\bsfour$ lifts to $\mathbb{T}^2\times I$, the orientable double cover of this space corresponds to $\mathbb{T}^3$, and thus admits either flat, $\Sol$ or $\Nil$ metric. 
\end{subcase}

%%%%%%%%%%%%%%%%%%%%%%%%%%%%%%%%%%%%%%%%%%%%%%%%%%%%%%%%%%%%%%%%%%%%

\end{case}

%%%%%%%%%%%%%%%%%%%%%%%%%%%%%%%%%%%%%%%%%%%%%%%%%%%%%%%%%%%%%%%%%%%%

\begin{case}[$\partial B=K^2$] In the remaining cases, the pieces we deal with are $\mathbb{S}^1\widetilde{\times} D^2$, $K^2\hat{\times} I$, $\bpt$ and $\bptwo$.

%%%%%%%%%%%%%%%%%%%%%%%%%%%%%%%%%%%%%%%%%%%%%%%%%%%%%%%%%%%%%%%%%%%%

\begin{subcase}[$(\mathbb{S}^1\widetilde{\times} D^2) \cup (\mathbb{S}^1\widetilde{\times} D^2)$]
\label{SUBCASE:D2tildetimesS1UD2tildetimesS1}
The space $\mathbb{S}^1\widetilde{\times} D^2$ is a solid Klein bottle, and as such, its orientable double cover is $\mathbb{S}^1\times D^2$ with the involution $\rho(x,y)=(-x,\overline{y})$.
Let $\varphi: \mathbb{S}^1\widetilde{\times} \mathbb{S}^1 \to \mathbb{S}^1\widetilde{\times} \mathbb{S}^1$ be the homeomorphism used to produce the gluing. 

Using standard theory for covering spaces, there is a lift of $\varphi$ that we denote $\widetilde{\varphi}:\mathbb{S}^1\times \mathbb{S}^1\to \mathbb{S}^1\times \mathbb{S}^1$; thus the orientable cover of our space is obtained by gluing two copies of $\mathbb{S}^1\times D^2$ by a homeomorphism of the torus, resulting on spaces admitting $\mathbb{S}^3$- or $\mathbb{S}^2\times \mathbb{R}$-geometries.
\end{subcase}

%%%%%%%%%%%%%%%%%%%%%%%%%%%%%%%%%%%%%%%%%%%%%%%%%%%%%%%%%%%%%%%%%%%%

\begin{subcase}[$(\mathbb{S}^1\widetilde{\times} D^2)\cup (K^2\hat{\times}I)$]
\label{SUBCASE:D2tildetimesS1UK2hattimesI}
As in the previous case, $\mathbb{S}^1\widetilde{\times}D^2 $ lifts to  $\mathbb{S}^1{\times}D^2$ in the orientable cover.  On the other hand, by the arguments outlined in the final part of Section \ref{S:PRELIM}, the orientable twofold cover of $K^2\hat{\times}I$ is $K^2\widetilde{\times}I$.
 Thus by Lemma \ref{L:DBL_COVER_GLUING}, the double cover of the space under examination is obtained gluing  $\mathbb{S}^1{\times}D^2$ to $K^2\widetilde{\times}I$ along their boundary. 
This case was examined in case \ref{Subcase:D2timesS1UK2tildeI}, where it was shown that the allowed geometries were $\mathbb{S}^3$ and $\mathbb{S}^2\times\mathbb{R}$. 
\end{subcase}

%%%%%%%%%%%%%%%%%%%%%%%%%%%%%%%%%%%%%%%%%%%%%%%%%%%%%%%%%%%%%%%%%%%%

\begin{subcase}[$(\mathbb{S}^1\widetilde{\times} D^2) \cup \bpt$]
 Both of the pieces have $D^2\times \mathbb{S}^1$ as orientable double cover. Then, by lifting the gluing  homeomorphism of $K^2$ to a homeomorphism of $T^2$ , we get that the space $\mathbb{S}^1\widetilde{\times} D^2 \cup \bpt$ is doubly covered by $D^2\times \mathbb{S}^1 \cup_{\partial} D^2\times \mathbb{S}^1$, that is,  a lens space or $\mathbb{S}^2\times \mathbb{S}^1$; the geometries can be $\mathbb{S}^3$ or $\mathbb{S}^2\times\mathbb{R}$.
\end{subcase}

%%%%%%%%%%%%%%%%%%%%%%%%%%%%%%%%%%%%%%%%%%%%%%%%%%%%%%%%%%%%%%%%%%%%

\begin{subcase}[$(\mathbb{S}^1\widetilde{\times} D^2) \cup \bptwo$]
\label{SUBCASE:D2tildetimesS1UBP2}

As in the previous cases, $\mathbb{S}^1\widetilde{\times} D^2$ has $\mathbb{S}^1\times D^2$ as orientable double cover. On the other hand, the orientable double branched cover of $\bptwo$ is $K^2\widetilde{\times}I$. 
One sees that as follows. Recall that $K^2\widetilde{\times}I$ is obtained from $T^2\times[-1,1]$ by the $\mathbb{Z}_2$-action induced by the involution
\[
T^2\times[-1,1]\to T^2\times[-1,1]
\]
\[
((z_1,z_2),t)\mapsto ((\overline{z}_1,-z_2,-t)).
\]
Observe that the quotient of $T^2\times\{0\}$ is the core Klein bottle in $K^2\widetilde{\times}I$.
Denote the points in $K^2\widetilde{\times}I$ by $[z_1,z_2,t]$ and consider the map
\[
\varphi:K^2\widetilde{\times}I\to K^2\widetilde{\times}I
\]
\[
[z_1,z_2,t]\mapsto [-z_1,\overline{z}_2,t].
\]
This map has two fixed points, namely $[i,i,0]$ and $[i,-i,0]$. Moreover, the map induces an involution on the core Klein bottle $K^2\times\{0\}$ of  $K^2\widetilde{\times}I$ with the preceding two points as fixed points. By the classification of involutions on $K^2$ (see \cite{Nat}), $K^2\times\{0\}/\varphi$ is isometric to an $\RP^2$ with a flat metric with two metric singularities. Comparing this construction with the construction in \cite[Corollary 2.56]{MitYam}, one sees that the double branched cover of $\bptwo$ is $K^2\widetilde{\times}I$.

It follows from standard covering spaces theory, as in Case \ref{SUBCASE:D2tildetimesS1UD2tildetimesS1}, that there is a lift of the gluing homeomorphism. Therefore, the orientable double cover of $(\mathbb{S}^1\widetilde{\times} D^2) \cup \bptwo$ is the gluing of a solid torus $\mathbb{S}^1\times D^2$ and $K^2\widetilde{\times}I$ via a homeomorphism of the boundary. Hence, by Case \ref{Subcase:D2timesS1UK2tildeI}, it admits either $\mathbb{S}^3$- or $(\mathbb{S}^2\times \mathbb{R})$-geometry.
\end{subcase}

%%%%%%%%%%%%%%%%%%%%%%%%%%%%%%%%%%%%%%%%%%%%%%%%%%%%%%%%%%%%%%%%%%%%

\begin{subcase}[$(K^2\hat{\times}I) \cup (K^2\hat{\times}I)$]
This is a \textit{Klein bottle semibundle}; these were classified in  \cite{GLHeiGA}. To check that the space is geometric, consider the oriented double cover, obtained by gluing two copies of $K^2\widetilde{\times} I$ along their boundaries. This corresponds to case \ref{SUBCASE:K2tildeIUK2tildeI}.
\end{subcase}

%%%%%%%%%%%%%%%%%%%%%%%%%%%%%%%%%%%%%%%%%%%%%%%%%%%%%%%%%%%%%%%%%%%%

\begin{subcase}[$(K^2\hat{\times}I) \cup\bpt$]
Recall that $K^2\hat{\times}I$ is doubly covered by $K^2\widetilde{\times} I$ while the oriented branched cover of $\bpt$ is $\mathbb{S}^1\times D^2$.
This situation was considered in case \ref{Subcase:D2timesS1UK2tildeI}.
\end{subcase}

%%%%%%%%%%%%%%%%%%%%%%%%%%%%%%%%%%%%%%%%%%%%%%%%%%%%%%%%%%%%%%%%%%%%

\begin{subcase}[$(K^2\hat{\times}I) \cup\bptwo$]
Both pieces have $K^2\widetilde{\times}I$ as orientable double cover and therefore, the orientable double cover of $(K^2\hat{\times}I) \cup\bptwo$ is a gluing of two copies of $K^2\widetilde{\times}I$ by a homeomorphism between the boundaries. Then, it follows from Case \ref{SUBCASE:K2tildeIUK2tildeI} that the space in question admits the $\mathbb{R}^3$-, $\Nil$- or $\Sol$-geometry depending of the gluing homeomorphism. 
\end{subcase}

%%%%%%%%%%%%%%%%%%%%%%%%%%%%%%%%%%%%%%%%%%%%%%%%%%%%%%%%%%%%%%%%%%%%

\begin{subcase}[$\bpt\cup \bpt$]
\label{SUBCASE:BPTUBPT}
The orientable double cover of $\bpt$ is a solid torus. Therefore, $\bpt\cup\bpt $ is doubly covered by a union of two solid tori along the boundary, i.e. a lens space.  
\end{subcase}

%%%%%%%%%%%%%%%%%%%%%%%%%%%%%%%%%%%%%%%%%%%%%%%%%%%%%%%%%%%%%%%%%%%%

\begin{subcase}[$\bpt\cup\bptwo$]
The orientable double cover of $\bpt$ is a solid torus $\mathbb{S}^1\times D^2$, while the orientable double cover of $\bptwo$ is $K^2\widetilde{\times}I$. Therefore, as in Case \ref{SUBCASE:D2tildetimesS1UBP2}, it follows from Case \ref{SUBCASE:D2tildetimesS1UK2hattimesI} that $\bpt\cup\bptwo$ admits the $\mathbb{S}^3$- or $(\mathbb{S}^2\times \mathbb{R})$-geometry.
\end{subcase}

%%%%%%%%%%%%%%%%%%%%%%%%%%%%%%%%%%%%%%%%%%%%%%%%%%%%%%%%%%%%%%%%%%%%

\begin{subcase}[$\bptwo\cup\bptwo$]
The space $\bptwo$ has $K^2\widetilde{\times}I$ as orientable double cover and therefore, the branched orientable double cover of  $\bptwo\cup\bptwo$ consists of a gluing of two copies of $K^2\widetilde{\times}I$ by a homeomorphism of the boundary. The space obtained by such construction was considered in Case \ref{SUBCASE:K2tildeIUK2tildeI}, from which it follows that $\bptwo\cup\bptwo$ admits the $\mathbb{R}^3$-, $\Nil$- or $\Sol$-geometry depending of the gluing. 
\end{subcase}

%%%%%%%%%%%%%%%%%%%%%%%%%%%%%%%%%%%%%%%%%%%%%%%%%%%%%%%%%%%%%%%%%%%%

\end{case}

%------------------------------------------------------
% 		SECTION: ZERO-DIMENSIONAL LIMIT SPACE
%------------------------------------------------------

\section{Zero-dimensional limit space}
\label{S:ZERO_DIM_LIMIT}
We now address the case in which the limit space $Y$ is zero-dimensional. From the classification of collapsing closed Alexandrov $3$-spaces \cite{MitYam}, summarized in Table \ref{TBL:0-DIM_LIMIT}, we see that there are three possibilities for the homeomorphism type of $X_i$.\\ 

In the case that $X_i$ is a generalized Seifert fiber space $\seif(Z)$ (with $\curv Z\geq 0$) possibly with attached generalized Solid Tori and Klein Bottles, it follows from our analysis in Section \ref{S:TWO_DIM_LIMIT} that $X_i$ is geometric. \\

The next possibility is that $X_i$ is homeomorphic to one of the spaces appearing in Section \ref{S:ONE_DIM_LIMIT}. Therefore, we can conclude from our previous analysis that $X_i$ is geometric.\\

Finally, if $X_i$ is a closed, non-negatively curved Alexandrov space, then it follows from \cite[Theorem 1.3]{GalGui} that $X_i$ is geometric. This concludes the proof of Theorem~\ref{THM:MAIN_THEOREM}.
\hfill $\square$

%------------------------------------------------------
%  BEGIN BIBLIOGRAPHY ---------------------------------
%------------------------------------------------------
\bibliographystyle{amsplain}

%------------------------------------------------------
% END BIBLIOGRAPHY ------------------------------------
%------------------------------------------------------

\end{document}